\documentclass[conference]{asl}
\usepackage{rotate}
\usepackage{xcolor}

\usepackage[
backend=biber,
style=alphabetic,
]{biblatex}
\usepackage{url}
\usepackage[shortlabels]{enumitem}

\newcommand{\bm}{\boldsymbol}
\newcommand{\R}{\mathbb{R}}
\newcommand{\T}{\mathcal{T}}
\newcommand{\Q}{\mathcal{Q}}
\newcommand{\G}{\mathcal{G}}
\newcommand{\I}{\mathcal{I}}
\newcommand{\M}{\mathcal{M}}
\newcommand{\U}{\mathcal{U}}
\newcommand{\N}{\mathcal{N}}
\newcommand{\C}{\mathcal{C}}

\newtheorem{theorem}{Theorem}[section]
\newtheorem{corollary}[theorem]{Corollary}
\newtheorem{lemma}[theorem]{Lemma}
\newtheorem{claim}[theorem]{Claim}
\newtheorem{subclaim}[theorem]{Subclaim}
\newtheorem{fact}[theorem]{Fact}
\newtheorem{definition}[theorem]{Definition}
\newtheorem{conjecture}[theorem]{Conjecture}
\newtheorem{remark}[theorem]{Remark} 

\def\confname{}
\def\copyrightyear{2023}

\title{Unreachability of Inductive-Like Pointclasses in $L(\R)$}

\author{Derek Levinson}
\revauthor{Levinson, Derek}

\address{Department of Mathematics\\ University of California\\
Los Angeles, CA 90095-1555, USA}
\email{djlevins@math.ucla.edu}
\urladdr{https://www.math.ucla.edu/ \urltilde djlevins/}
\orcid{0009-0006-3647-1666}

\author{Itay Neeman}
\revauthor{Neeman, Itay}
\address{Department of Mathematics\\ University of California\\
Los Angeles, CA 90095-1555, USA}
\email{ineeman@math.ucla.edu}
\urladdr{https://www.math.ucla.edu/ \urltilde ineeman/}

\author{Grigor Sargsyan}
\revauthor{Sargsyan, Grigor}
\address{Institute of Mathematics\\ Polish Academy of Sciences \\ Gdansk, Poland}
\email{sargsyan@impan.pl}
\urladdr{https://grigorsarg.github.io/}
\orcid{0009-0009-7639-3940}

\begin{document}

\begin{abstract}
In \cite{twoapp}, Hjorth proved from $ZF + AD + DC$ that there is no sequence of distinct $\bm{\Sigma^1_2}$ sets of length $\bm{\delta^1_2}$. \cite{hra} extends Hjorth's technique to show there is no sequence of distinct $\bm{\Sigma^1_{2n}}$ sets of length $\bm{\delta^1_{2n}}$. Sargsyan conjectured an analogous property is true for any regular Suslin pointclass in $L(\R)$ --- i.e. if $\kappa$ is a regular Suslin cardinal in $L(\R)$, then there is no sequence of distinct $\kappa$-Suslin sets of length $\kappa^+$ in $L(\R)$. We prove this in the case that the pointclass $S(\kappa)$ is inductive-like.
\end{abstract}

\maketitle

\section{Introduction}
\label{intro chapter}

\begin{definition}
    For a boldface pointclass $\bm{\Gamma}$, we say $\lambda$ is $\bm{\Gamma}$-reachable if there is a sequence of distinct $\bm{\Gamma}$ sets of length $\lambda$ and $\lambda$ is $\bm{\Gamma}$-unreachable if $\lambda$ is not $\bm{\Gamma}$-reachable.
\end{definition}

The problem of unreachability is to determine the minimal $\lambda$ which is $\bm{\Gamma}$-unreachable for each pointclass $\bm{\Gamma}$. As this problem is trivial assuming the axiom of choice, unreachability is exclusively studied under determinacy assumptions. Under $AD$, unreachability yields an interesting measure of the complexity of a pointclass. An early result in this area is Harrington's theorem that there is no injection of $\omega_1$ into any pointclass strictly below the pointclass of Borel sets in the Wadge hierarchy (see \cite{Harrington}).

\begin{theorem}[Harrington]
    If $\beta < \omega_1$, then $\omega_1$ is $\bm{\Pi^0_\beta}$-unreachable.
\end{theorem}

A recent application of Harrington's theorem was the resolution of the decomposability conjecture by Marks and Day (see \cite{decomposability}).

Prior work on unreachability has focused on levels of the projective hierarchy. Kechris gave a lower bound on the complexity of the pointclass needed to reach $\bm{\delta^1_{2n+2}}$ (see \cite{Kechris}).

\begin{theorem}[Kechris]
    Assume $ZF + AD + DC$. $\bm{\delta^1_{2n+2}}$ is $\bm{\Delta^1_{2n+1}}$-unreachable.
\end{theorem}

In \cite{Kechris}, Kechris conjectured his own result could be strengthened to $\bm{\delta^1_{2n+2}}$ is $\bm{\Delta^1_{2n+2}}$-unreachable. He also made a second, stronger conjecture that $\bm{\delta^1_{2n+2}}$ is $\bm{\Sigma^1_{2n+2}}$-unreachable. Jackson proved the former in \cite{Jackson}.

\begin{theorem}[Jackson]
    Assume $ZF + AD + DC$. $\bm{\delta^1_{2n+2}}$ is $\bm{\Delta^1_{2n+2}}$-unreachable.
\end{theorem}

\cite{Jackson} also made progress on Kechris's second conjecture by showing there is no strictly increasing sequence of $\bm{\Sigma^1_{2n+2}}$ sets of length $\bm{\delta^1_{2n+2}}$. In fact, Jackson and Martin proved the following more general theorem.

\begin{theorem}[Jackson]
\label{no strictly inc}
    Assume $ZF +AD +DC$. Suppose $\kappa$ is a Suslin cardinal, and $\kappa$ is either a successor cardinal or a regular limit cardinal. Then there is no strictly increasing (or strictly decreasing) sequence $\langle A_\alpha : \alpha < \kappa^+\rangle$ contained in $S(\kappa)$.
\end{theorem}

But the resolution of Kechris's second conjecture eluded the traditional techniques of descriptive set theory. Hjorth pioneered the use of inner model theory in this area to resolve one case of Kechris's second conjecture (see \cite{twoapp}).

\begin{theorem}[Hjorth]
\label{hjorth result}
    Assume $ZF + AD + DC$. $\bm{\delta^1_{2}}$ is $\bm{\Sigma^1_{2}}$-unreachable.
\end{theorem}

Kechris also pointed out the following corollary of Hjorth's result.

\begin{corollary}
    Assume $ZF + AD + DC$. A $\bm{\Pi^1_2}$ equivalence relation has either $2^{\aleph_0}$ or $\leq \aleph_1$ equivalence classes.
\end{corollary}

Hjorth's proof of Theorem \ref{hjorth result} involved an application of the Kechris-Martin Theorem, which precluded an easy generalization of his technique to other projective pointclasses. The rest of Kechris's second conjecture survived another two decades, until Sargsyan found a modification of Hjorth's proof which generalized to the rest of the projective hierarchy (see \cite{hra}).

\begin{theorem}[Sargsyan]
\label{projective thm}
    Assume $ZF + AD + DC$. $\bm{\delta^1_{2n+2}}$ is $\bm{\Sigma^1_{2n+2}}$-unreachable.
\end{theorem}

The following result of Kechris shows Sargsyan's theorem is optimal.

\begin{theorem}[Kechris]
\label{strictly inc sequence}
    Assume $ZF + AD + DC$. Suppose $\kappa$ is a Suslin cardinal. Then there is a strictly increasing sequence $\langle A_\alpha : \alpha < \kappa \rangle$ contained in $S(\kappa)$.
\end{theorem}

Sargsyan's theorem resolves the problem of unreachability for every level of the projective hierarchy. He conjectured an analogous result holds for every regular Suslin pointclass.

\begin{conjecture}[Sargsyan]
\label{Grigor conjecture}
    Assume $AD^+$. Suppose $\kappa$ is a regular Suslin cardinal. Then $\kappa^+$ is $S(\kappa)$-unreachable.
\end{conjecture}

Below, we prove part of Conjecture \ref{Grigor conjecture}.

\begin{theorem}
\label{main thm v2}
    Assume $ZF + AD + DC + V=L(\mathbb{R})$. Suppose $\kappa$ is a regular Suslin cardinal and $S(\kappa)$ is inductive-like. Then $\kappa^+$ is $S(\kappa)$-unreachable.
\end{theorem}

$ZF + AD + DC + V=L(\mathbb{R})$ implies $AD^+$, so Theorem \ref{main thm v2} is a special case of Conjecture \ref{Grigor conjecture}. Theorem \ref{strictly inc sequence} demonstrates this is the optimal result for inductive-like pointclasses.

Let $\bm{\Gamma} = S(\kappa)$ for $\kappa$ as in Theorem \ref{main thm v2}. Then $\kappa = \bm{\delta_\Gamma}$. $ZF + AD + DC + V=L(\mathbb{R})$ also implies any inductive-like pointclass $\bm{\Gamma}$ is of the form $S(\kappa)$ for some regular Suslin cardinal $\kappa$. So an equivalent formulation of Theorem \ref{main thm v2} is the following:

\begin{theorem}
\label{main thm}
    Assume $ZF + AD + DC + V=L(\mathbb{R})$. Suppose $\bm{\Gamma}$ is an inductive-like pointclass. Then $\bm{\delta_{\Gamma}^+}$ is $\bm{\Gamma}$-unreachable.
\end{theorem}

Our proof of Theorem \ref{main thm} extends the inner model theory approach pioneered in \cite{twoapp}. Our technique also gives an alternative proof of Theorem \ref{projective thm}.

\section{Background}
We will assume the reader is familiar with the basics of descriptive set theory espoused in \cite{mosdst} and the theory of iteration strategies for premice covered in \cite{ooimt}. The rest of the necessary background is covered below. In Section \ref{dst in L(R)}, we summarize Steel's classification of the scaled pointclasses and Suslin pointclasses in $L(\R)$. Section \ref{woodin and iteration section} reviews the relationship between Woodin cardinals and iteration trees. Two inner model constructions are covered in Sections \ref{m-s construction section} and \ref{s-const section}. In Section \ref{suitable mice section}, we review results from the core model induction demonstrating the existence of mice corresponding to inductive-like pointclasses in $L(\R)$.\\

\subsection{The Pointclasses of $L(\R)$}
\label{dst in L(R)}

We will assume for this section $ZF + DC + AD + V = L(\R)$. All of the results in this section are due to Steel and are proven outright or else implicit in \cite{scales}.

The boldface pointclasses we are interested in all appear in a hierarchy we will now define. If $\bm{\Gamma}$ and $\bm{\Lambda}$ are non-selfdual pointclasses, say $\{\bm{\Gamma},\bm{\Gamma^c}\} <_w \{\bm{\Lambda}, \bm{\Lambda^c}\}$ if $\bm{\Gamma} \subset \bm{\Lambda} \cap \bm{\Lambda^c}$. This is a wellordering by Wadge's Lemma. For $\alpha < \Theta$, consider the $\alpha$th pair $\{\bm{\Gamma},\bm{\Gamma^c}\}$ in this wellordering such that $\bm{\Gamma}$ or $\bm{\Gamma^c}$ is closed under projection. Let $\bm{\Sigma^1_\alpha}$ denote whichever of the two is closed under projection --- if both are, $\bm{\Sigma^1_\alpha}$ denotes whichever has the separation property. Let $\bm{\Pi^1_\alpha} = (\bm{\Sigma^1_\alpha})^c$.

For any pointclass $\bm{\Gamma}$, we define 
\begin{align*}
    &\bm{\Delta_\Gamma} = \bm{\Gamma} \cap \bm{\Gamma^c} \text{ and} \\
    &\bm{\delta_\Gamma} = sup \{|\leq^*| :\, \leq^* \text{ is a prewellordering in } \bm{\Delta_\Gamma}\}.
\end{align*}.

Let $\bm{\delta^1_\alpha} = \bm{\delta_{\Sigma^1_\alpha}}$. The pointclasses $\{\bm{\Sigma^1_n}: n\in\omega \}$ and $\{\bm{\Pi^1_n}: n\in\omega \}$ are the usual levels of the projective hierarchy. We will refer to the collection of pointclasses $\{\bm{\Sigma^1_\alpha}: \alpha\in ON \} \cup \{\bm{\Pi^1_\alpha}: \alpha \in ON\}$ as the extended projective hierarchy.

We now define a hierarchy slightly coarser than the one above. If $n\in \omega$ and $\alpha\in ON$, we say a pointset $A$ is in the pointclass $\bm{\Sigma_n}(J_\alpha(\R))$ if there is a $\Sigma_n$ formula $\phi$ with real parameters such that $A = \{x : J_\alpha(\R) \models \phi[x]\}$. $\bm{\Pi_n}(J_\alpha(\R))$ is defined analogously with $\Pi_n$-formulas.\footnote{See \cite{scales} for the definition of $J_\alpha(\R)$. Alternatively, the reader will not lose too much of importance by pretending $J_\alpha(\R) = L_\alpha(\R)$.} The Levy hierarchy consists of all pointclasses of the form $\bm{\Sigma_n}(J_\alpha(\R))$ or $\bm{\Pi_n}(J_\alpha(\R))$ for some $n$ and $\alpha$. It is clear any pointclass in the Levy hierarchy equals $\bm{\Sigma^1_\alpha}$ or $\bm{\Pi^1_\alpha}$ for some $\alpha$, but the converse is false.

In this section, we will classify the scaled pointclasses within the Levy hierarchy, relate the Levy hierarchy to the extended projective hierarchy, and classify the regular Suslin pointclasses.\\

\subsubsection{Classification of Scaled Pointclasses}

A $\Sigma_1$-gap is a maximal interval $[\alpha,\beta]$ such that for any real $x$, the $\Sigma_1$-theory of $x$ is the same in $J_\alpha(\R)$ and $J_\beta(\R)$.

We say the gap $[\alpha,\beta]$ is admissible if $J_\alpha(\R)\models KP$, equivalently, if the pointclass $\bm{\Sigma_1}(J_\alpha(\R))$ is closed under coprojection. Suppose $[\alpha,\beta]$ is an admissible gap. Let $n_\beta\in\mathbb{N}$ be least such that the pointclass $\bm{\Sigma_{n_\beta}}(J_\beta(\R))$ is not contained in $J_\beta(\R)$. We say $[\alpha,\beta]$ is a strong gap if for any $b\in J_\beta(\R)$, there is $\beta' < \beta$ and $b' \in J_{\beta'}(\R)$ such that the $\Sigma_{n_\beta}$ and $\Pi_{n_\beta}$ theories of $b'$ in $J_{\beta'}(\R)$ are the same as the $\Sigma_{n_\beta}$ and $\Pi_{n_\beta}$ theories of $b$ in $J_{\beta}(\R)$. Otherwise, we say $[\alpha,\beta]$ is weak.

\begin{theorem}
\label{scales classification}
    Suppose $\bm{\Gamma}$ is a pointclass in the Levy hierarchy. If $\bm{\Gamma}$ is scaled, then one of the following holds.
    \begin{enumerate}
        \item $\bm{\Gamma} = \bm{\Sigma}_{2k+1}(J_\alpha(\R))$ for some $k\in \omega$ and some $\alpha$ beginning an inadmissible gap.
        \item $\bm{\Gamma} = \bm{\Pi}_{2k+2}(J_\alpha(\R))$ for some $k\in \omega$ and some $\alpha$ beginning an inadmissible gap.
        \item $\bm{\Gamma} = \bm{\Sigma}_1(J_\alpha(\R))$ for some $\alpha$ beginning an admissible gap.
        \item $\bm{\Gamma} = \bm{\Sigma}_{n_{\beta}+2k}(J_\beta(\R))$ for some $k\in \omega$ and some $\beta$ ending a weak gap.
        \item $\bm{\Gamma} = \bm{\Pi}_{n_{\beta}+2k + 1}(J_\beta(\R))$ for some $k\in \omega$ and some $\beta$ ending a weak gap.
    \end{enumerate}
\end{theorem}

\begin{definition}
    A self-justifying system (sjs) is a countable set $\mathcal{B} \subseteq \mathcal{P}(\mathbb{R})$ which is closed under complements and has the property that every $B\in\mathcal{B}$ admits a scale $\vec{\psi}$ such that $\leq_{\psi_n} \in \mathcal{B}$ for all $n$.
\end{definition}

\begin{definition}
    Let $z\in\R$ and $\gamma\in ON$. $OD^{<\gamma}(z)$ is the set of pointsets which are ordinal definable from the parameter $z$ in $J_\xi(\R)$ for some $\xi<\gamma$. $OD^{<\gamma}$ denotes $OD^{<\gamma}(0)$.
\end{definition}

The proof of Theorem \ref{scales classification} also gives:

\begin{theorem}
\label{where sjs appears}
    Suppose $[\alpha,\beta]$ is an admissible gap. Let $\beta'$ be the least ordinal such that there is a scale for a universal $\bm{\Pi_1}(J_\alpha(\R))$-set definable over $J_{\beta'}(\R)$. Then there is $z\in\R$ and a sjs $\mathcal{B} \subset OD^{<\beta'}(z)$ such that a universal $\bm{\Pi_1}(J_\alpha(\R))$-set is in $\mathcal{B}$ and either
    \begin{enumerate}
        \item $[\alpha,\beta]$ is weak and $\beta' = \beta$ or
        \item $[\alpha,\beta]$ is strong and $\beta' = \beta+1$.
    \end{enumerate}
\end{theorem}

\begin{remark}
\label{form of ind-like}
    Suppose $\bm{\Gamma}$ is a boldface inductive-like pointclass in $L(\R)$. Then
    \begin{enumerate}
        \item $\bm{\Gamma} = \bm{\Sigma_1}(J_\alpha(\R))$ for some $\alpha$ beginning an admissible gap,
        \item there is $x\in \R$ such that letting $\Gamma$ be the class of pointsets which are $\Sigma_1$-definable over $J_\alpha(\R)$ from the parameter $x$, $\bm{\Gamma}$ is the closure of $\Gamma$ under preimages by continuous functions, and
        \item \label{Gamma is (Sigma^2_1)^Delta} $\bm{\Gamma} = (\bm{\Sigma^2_1})^{\bm{\Delta_\Gamma}}$.\footnote{We say $A\subseteq \R$ is in $(\bm{\Sigma^2_1})^{\bm{\Delta_\Gamma}}$ if there is $z\in \R$ and a formula $\phi$ such that for all $x\in \R$, $x\in A \iff (\exists B\in \bm{\Delta_\Gamma}) (\R,B) \models \phi(x,z)$.}
    \end{enumerate}
\end{remark}
\vspace{2mm}
\subsubsection{Relationship between the Levy Hierarchy and the Extended Projective Hierarchy}

\begin{definition}
    Suppose $\lambda < \Theta$ is a limit ordinal. We say
    \begin{itemize}
        \item $\lambda$ is type I if $\bm{\Sigma^1_\lambda}$ is closed under finite intersection but not countable intersection,
        \item $\lambda$ is type II if $\bm{\Sigma^1_\lambda}$ is not closed under finite intersection,
        \item $\lambda$ is type III if $\bm{\Sigma^1_\lambda}$ is closed under countable intersection but not coprojection, and
        \item $\lambda$ is type IV if $\bm{\Sigma^1_\lambda}$ is closed under coprojection.
    \end{itemize}
    
\end{definition}

Let $\langle \delta_\alpha : \alpha < \Theta \rangle$ enumerate the ordinals $\delta$ such that there exist sets of reals in $J_{\delta+1}(\R) \backslash J_\delta(\R)$. Let $n_\alpha$ be minimal such that $\bm{\Sigma_{n_\alpha}}(J_{\delta_\alpha}(\R)) \not\subset J_{\delta_\alpha}(\R)$.

\begin{theorem}
\label{extended projective hierarchy to levy hierarchy}
    Suppose $\alpha < \Theta$.
    \begin{enumerate}
        \item If $\omega\alpha$ is type I, then $\bm{\Sigma^1_{\omega\alpha + k}} = \bm{\Sigma_{n_\alpha+k}}(J_{\delta_\alpha}(\R))$ for all $k\in \omega$.
        \item If $\omega\alpha$ is type II or III, then $\bm{\Sigma^1_{\omega\alpha + k + 1}} = \bm{\Sigma_{n_\alpha+k}}(J_{\delta_\alpha}(\R))$ for all $k\in \omega$.
        \item If $\omega\alpha$ is type IV, then $\bm{\Pi^1_{\omega\alpha}} = \bm{\Sigma_{n_\alpha}}(J_{\delta_\alpha}(\R))$ 
        and \\$\bm{\Sigma^1_{\omega\alpha + k+1}} = \bm{\Sigma_{n_\alpha+k}}(J_{\delta_\alpha}(\R))$ for all $k\in \omega\backslash\{0\}$.
    \end{enumerate}        
\end{theorem}
\vspace{2mm}
\subsubsection{Classification of Suslin Pointclasses}

There is a related classification of the Suslin pointclasses. For $\alpha < \Theta$, let $\kappa_\alpha$ be the $\alpha$th Suslin cardinal. Let $\nu_\alpha$ be the $\alpha$th ordinal $\nu$ such that $\bm{\Sigma^1_\nu}$ or $\bm{\Pi^1_\nu}$ is scaled.

\begin{theorem}
\label{classification of suslin pointclasses}
    Let $\lambda < \bm{\delta^2_1}$ be a limit cardinal and $\nu = sup\{\nu_\alpha : \alpha < \lambda\}$.
    \begin{enumerate}
        \item If $\nu$ is type I, then for all $k\in \omega$
        \begin{itemize}
            \item $\bm{\Sigma^1_{\nu+2k}}$ and $\bm{\Pi^1_{\nu+2k+1}}$ are scaled,
            \item $S(\kappa_{\lambda + k}) = \bm{\Sigma^1_{\nu+k+1}}$,
            \item $\kappa_{\lambda+2k+1} = \bm{\delta^1_{\nu+2k+1}} = (\kappa_{\lambda+2k})^+$, and
            \item $cof(\kappa_{\lambda+2k}) = \omega$.
        \end{itemize}
        \item If $\nu$ is type II or III, then for all $k\in \omega$
        \begin{itemize}
            \item $\bm{\Sigma^1_{\nu+2k+1}}$ and $\bm{\Pi^1_{\nu+2k}}$ are scaled,
            \item $S(\kappa_{\lambda + k}) = \bm{\Sigma^1_{\nu+k+1}}$,
            \item $\kappa_{\lambda+2k+2} = \bm{\delta^1_{\nu+2k+2}} = (\kappa_{\lambda+2k+1})^+$, and
            \item $cof(\kappa_{\lambda+2k+1}) = \omega$.
        \end{itemize}
        \item \label{ind-like case of suslin classification} If $\nu$ is type IV, then $\bm{\Pi^1_\nu}$ is scaled, $S(\kappa_\lambda) = \bm{\Pi^1_\nu}$, and for all $k\in \omega$, letting $\mu = \nu_{\lambda+1}$,
        \begin{itemize}
            \item $\bm{\Sigma^1_{\mu+2k}}$ and $\bm{\Pi^1_{\mu+2k+1}}$ are scaled,
            \item $S(\kappa_{\lambda + k + 1}) = \bm{\Sigma^1_{\mu+k+1}}$,
            \item $\kappa_{\lambda+2k+2} = \bm{\delta^1_{\mu+2k+1}} = (\kappa_{\lambda+2k+1})^+$, and
            \item $cof(\kappa_{\lambda+2k+1}) = \omega$.
        \end{itemize}
    \end{enumerate}
\end{theorem}

\begin{corollary}
\label{classification of regular Suslins}
    Suppose $\bm{\Gamma} = S(\kappa)$ for a regular Suslin cardinal $\kappa \leq \bm{\delta^2_1}$. Then one of the following holds.
    \begin{enumerate}
        \item $\bm{\Gamma} = \bm{\Sigma}_{2k+1}(J_\alpha(\R))$ for some $k\in \omega$ and some $\alpha$ beginning an inadmissible gap.
        \item $\bm{\Gamma} = \bm{\Sigma}_1(J_\alpha(\R))$ for some $\alpha$ beginning an admissible gap.
        \item $\bm{\Gamma} = \bm{\Sigma}_{n_{\beta}+2k}(J_\beta(\R))$ for some $k\in \omega$ and some $\beta$ ending a weak gap.
    \end{enumerate}
\end{corollary}
\vspace{2mm}
\subsection{Woodin Cardinals and Iterations}
\label{woodin and iteration section}

We borrow most of the notation of premice and iteration trees from \cite{ooimt}. In addition to the lightface premice defined in \cite{ooimt}, we will also consider premice built over some $a\in HC$. We write an $a$-premouse as $M = (J_\alpha^{\vec{E}},\in,\vec{E}\upharpoonright\alpha,E_\alpha,a)$, for a fine extender sequence $\vec{E} = \langle E_\eta: \eta \leq \alpha\rangle$. If $\beta\leq \alpha$, $M|\beta$ represents the premouse $(J_\beta^{\vec{E}},\in,\vec{E}\upharpoonright\beta,E_\beta,a)$. Unless otherwise specified, an iteration strategy will refer to an $(\omega_1,\omega_1)$-iteration strategy and a mouse will refer to a premouse with such a strategy. Under $ZF + AD$, $\omega_1$ is measurable, so an $(\omega_1,\omega_1)$-iteration strategy induces an $(\omega_1,\omega_1+1)$-iteration strategy. In particular, the theorems in this section requiring an $\omega_1+1$-iteration strategy will all apply to the mice we use in Section \ref{ind-like chapter}.

Additionally, if $\T$ is an iteration tree of limit length and $b$ is a cofinal, non-dropping branch through $\T$, we let $M_b^\T$ be the direct limit of the models on $b$ and let $i_b^\T: M_0^\T \to M_b^\T$ be the associated direct limit embedding.

For a model $M$, let $\delta_M$ denote the least Woodin cardinal of $M$ (if one exists) and $Ea_M$ denote Woodin's extender algebra in $M$ at $\delta_M$. Let $\kappa_M$ be the least cardinal of $M$ which is $<\delta_M$-strong in $M$. $ea$ will refer to the generic over $Ea_M$. When considering the product extender algebra $Ea_M \times Ea_M$, we will write $ea_l \times ea_r$ for the generic. $ea_r$ will typically code a pair which we shall write $(ea^1_r,ea^2_r)$. For posets of the form $Col(\omega,X)$, $\dot{g}$ denotes a name for the generic.

Suppose $M$ is a premouse with iteration strategy $\Sigma$. We say $N$ is a complete iterate of $M$ if $N$ is the last model of an iteration tree $\T$ on $M$ such that $\T$ is according to $\Sigma$ and the branch through $\T$ from $M$ to $N$ is non-dropping.\footnote{This is a slight abuse of notation, since being ``a complete iterate of $M$'' is dependent on $\Sigma$ as well as $M$. This will not cause any ambiguity, since the mice we are interested in have unique iteration strategies.}

\begin{theorem}
\label{genericity iteration for ext alg}
    Let $M$ be a countable premouse with an $\omega_1+1$-iteration strategy such that $M \models$ ``There is a Woodin cardinal.'' Then $Ea_M$ is a $\delta_M$-c.c. Boolean algebra and for any $x\in \R$, there is a countable, complete iterate $N$ of $M$ such that $x$ is $Ea_N$-generic over $N$.
\end{theorem}

\begin{corollary}
    \label{genericity iteration for Col}
    Let $M$ be a countable premouse with an $\omega_1+1$-iteration strategy such that $M \models$ ``There is a Woodin cardinal.''  Then for any $x\in\R$, there is a countable, complete iterate $N$ of $M$ and $g$ which is $Col(\omega,\delta_N)$-generic over $N$ such that $x\in N[g]$.
\end{corollary}

See Section 7.2 of \cite{ooimt} for a proof of Theorem \ref{genericity iteration for ext alg} and its corollary.

\begin{definition}
    For $\kappa <\delta$ and $A\subseteq \delta$, we say $\kappa$ is $A$-reflecting in $\delta$ if for every $\nu < \delta$, there is an extender $E$ with critical point $\kappa$ such that $i_E(\kappa)>\nu$ and $i_E(A) \cap \nu = A \cap \nu$.
\end{definition}

\begin{theorem}
\label{branch uniqueness}
    Suppose $b$ and $c$ are distinct wellfounded branches of a normal iteration tree $\T$ and $A\subseteq \delta(\T)$ is in $M^\T_b \cap M^\T_c$. Then there is $\kappa < \delta(\T)$ such that $M^\T_b \models$ ``$\kappa$ is $A$-reflecting in $\delta(\T)$,'' and this is witnessed by a sequence of extenders on the extender sequence of $\mathcal{M}(\T)$.
\end{theorem}

See 6.9 and 6.10 of \cite{ooimt} for  definitions of $\delta(\T)$ and $\M(\T)$ and a proof of Theorem \ref{branch uniqueness}. The theorem justifies the following definitions.

\begin{definition}
    Suppose $b$ is a wellfounded branch through a normal iteration tree $\T$. Let $\Q(b,\T)$ be the least initial segment of $M^\T_b$ extending $\M(\T)$ such that there is $A\subset\delta(\T)$ which is definable over $\Q(b,\T)$ and realizes $\delta(\T)$ is not Woodin via extenders in $\M(\T)$, if such an initial segment exists.
\end{definition}

\begin{definition}
    Suppose $M$ is a premouse and $\eta\in M$. We say $\eta$ is a cutpoint of $M$ if there is no extender on the fine extender sequence of $M$ with critical point less than $\eta$ and length greater than $\eta$. $\eta$ is a strong cutpoint if there is no extender on the fine extender sequence of $M$ with critical point less than or equal to $\eta$ and length greater than $\eta$.
\end{definition}

\begin{definition}
    Suppose $\T$ is a normal iteration tree. Let $\Q(\T)$ be the least $\delta(\T)$-sound, $\omega_1+1$-iterable premouse extending $\M(\T)$ and projecting to $\delta(\T)$ such that $\delta(\T)$ is a strong cutpoint of $\Q(\T)$ and there is $A\subset \delta(\T)$ which is definable over $\Q(\T)$ and realizes $\delta(\T)$ is not Woodin via extenders in $\M(\T)$, if one exists.
\end{definition}

It follows from Theorem \ref{branch uniqueness} that there is at most one wellfounded branch $b$ through $\T$ such that $\Q(\T)\trianglelefteq M^\T_b$. In many cases, we will be able to locate the branch a strategy $\Sigma$ chooses as the unique branch which absorbs $\Q(\T)$ in this sense.

Note an $\omega_1$-iteration strategy on a countable premouse can be coded by a set of reals. For $a\in HC$ and a pointclass $\Gamma$, this allows us to define
\begin{align*}
    Lp^\Gamma(a) = \, \bigcup\{N :\, &N \text{ is an } \omega\text{-sound } a\text{-premouse projecting to } a \\
    & \text{ with an $\omega_1$-iteration strategy in } \bm{\Delta_{\Gamma}}\}.
\end{align*}
$Lp^\Gamma(a)$ can be reorganized as an $a$-premouse, which is what we will typically use $Lp^\Gamma(a)$ to refer to.

Closely related to $Lp^\Gamma$ is the operator $C_\Gamma$. For $x\in \R$, 
\begin{align*}
    C_\Gamma(x) = \{z\in \R: z \text{ is } \Delta_\Gamma(x) \text{ in some countable ordinal}\}.
\end{align*}
And for $a\in HC$,
\begin{align*}
    C_\Gamma(a) = \{b\subseteq a: \text{for all reals } x \text{ coding } a,\, b_x\in C_\Gamma(x)\}.
\end{align*}

Here $b_x$ codes $b$ relative to $x$. See \cite{steel2016} for more details.

\begin{theorem}
\label{lp = c}
    Assume $AD^{L(\R)}$. Suppose $\Gamma$ is a (lightface) inductive-like pointclass in $L(\R)$ and $a\in HC$. Then $C_\Gamma(a) = Lp^\Gamma(a) \cap P(a)$.
\end{theorem}

\begin{remark}
\label{lp(a) contained in lp(b)}
    Suppose $a$ and $b$ are countable, transitive sets and $a\in b$. It is easy to see from the definition of $C_\Gamma$ that $C_\Gamma(a)\subseteq C_\Gamma(b)$. This, and the theorem above, implies $Lp^\Gamma(a)\subseteq Lp^\Gamma(b)$.
\end{remark}
\vspace{2mm}
\subsection{The Mitchell-Steel Construction}
\label{m-s construction section}

We shall require a method of building an $a$-premouse inside a premouse $M$ which contains $a$. Our main tool for this purpose is the fully backgrounded Mitchell-Steel construction developed in \cite{msbook}. This section reviews the construction and its properties.

We say a premouse $M$ is reliable if $\C_\omega(M)$ exists and is universal and solid. As we shall see in a moment, we will end the Mitchell-Steel construction if we reach a premouse which is not reliable. \cite{msbook} defines reliable to include the stronger property that $\C_\omega(M)$ is iterable. But the weaker properties of universality and solidity are enough to propagate the construction, and our weaker requirement ensures the construction does not end prematurely when performed inside a mouse. The definitions of universality and solidity can be found in \cite{ooimt}. In all of the cases relevant to us, universality and solidity are guaranteed and the reader will lose little by taking on faith that the construction does not end.

For the moment we will work in $V$ and assume $ZFC$. Fix $z\in\R$. Define a sequence of $z$-premice $\langle \M_\xi : \xi \in On\rangle$ inductively as follows.
\begin{enumerate}
    \item $\M_0 = (V_\omega,\in,\emptyset,\emptyset,z)$
    \item \label{active case of m-s construction} Suppose we have constructed $\M_\xi = (J_\alpha^{\vec{E}},\in,\vec{E},\emptyset,z)$. Note $\M_\xi$ is a passive premouse. Suppose also there is an extender $F^*$ over $V$, an extender $F$ over $\M_\xi$, and $\nu<\alpha$ such that 
    \begin{enumerate}
        \item $V_{\nu+\omega} \subset Ult(V,F^*)$,
        \item $\nu$ is the support of $F$,
        \item $F\upharpoonright \nu = F^*\cap ([\nu]^{<\omega} \times \M_\xi)$, and
        \item $\N_{\xi+1} = (J_\alpha^{\vec{E}},\in,\vec{E},F,z)$ is a premouse.
    \end{enumerate}
    If $\N_{\xi+1}$ is reliable, let $\M_{\xi+1} = \C_\omega(\N_{\xi+1})$. Otherwise, the construction ends. If there are multiple such $F^*$, we pick one which minimizes the support of $F$. We say $F^*$ is the extender used as a background at step $\xi+1$.
    
    \item Suppose we have constructed $\M_\xi = (J_\alpha^{\vec{E}},\in,\vec{E},E_\alpha,z)$ and either $\M_\xi$ is active or $\M_\xi$ is passive and there is no extender $F^*$ as above. Let $\N_{\xi+1} = (J_{\alpha+1}^{\vec{E}^\frown E_\alpha}, \in,\vec{E}^\frown E_\alpha,\emptyset,z)$. If $\N_{\xi+1}$ is reliable, let $\M_{\xi+1} = \C_\omega(\N_{\xi+1})$. Otherwise, the construction ends.
    
    \item Suppose we have constructed $\langle \M_\xi : \xi < \lambda\rangle$ for $\lambda$ a limit ordinal. Let \\$\eta = lim\,inf_{\xi < \lambda} (\rho_\omega(\M_\xi)^+)^{\M_\xi}$. Let $\N_\lambda$ be the passive premouse of height $\eta$ such that $\N_\lambda|\beta = lim_{\xi<\lambda}\, \M_\xi|\beta$ for all $\beta < \eta$. If $\N_\lambda$ is reliable, let $\M_\lambda = \C_\omega(\N_\lambda)$. Otherwise, the construction ends.
\end{enumerate}

Suppose the construction never breaks down. That is, $\M_\xi$ is defined for all $\xi\in On$.

\begin{theorem}
\label{levels of m-s construction projectum fact}
    Suppose $\zeta_0$ and $\xi$ are ordinals such that $\zeta_0 < \xi$ and $\kappa = \rho_\omega(\M_\xi) \leq \rho_\omega(\M_\zeta)$ for all $\zeta \geq \zeta_0$. Then $\M_\xi \trianglelefteq \M_\eta$ for all $\eta \geq \xi$. Moreover, $\M_{\xi+1} \models$ ``every set has cardinality at most $\kappa$.''
\end{theorem}

Let $\M$ be the class-sized model such that whenever $\xi\in On$ satisfies $\M_\xi \trianglelefteq \M_\eta$ for all $\eta\geq \xi$, $\M_\xi$ is an initial segment of $\M$. We call $\M$ the output of the Mitchell-Steel construction over $z$. For $\delta\in On$, we call $\M_\delta$ the output of the Mitchell-Steel construction of length $\delta$ over $z$.

\begin{theorem}
\label{woodin in m-s}
    Assume $ZFC$. Suppose $\delta$ is the least ordinal such that $\delta$ is Woodin in $L(V_\delta)$. Suppose the Mitchell-Steel construction in $V_\delta$ does not break down, and let $\M$ be the output of the construction. Then $\delta$ in Woodin in $L(\M)$.
\end{theorem}

See the proof of Theorem 11.3 of \cite{msbook}.

\begin{theorem}[Universality]
\label{universality of m-s}
    Assume $ZFC$. Let $\delta$ be Woodin and $z\in \R$. Assume the Mitchell-Steel construction of length $\delta$ over $z$ does not break down. Let $N$ be the output of the construction. Suppose no initial segment of $N$ satisfies ``there is a
    superstrong cardinal.'' Let $W$ be a premouse over x of height $\leq \delta$, and suppose $P$ 
    and $Q$ are the final models above $W$ and $N$, respectively, in a successful coiteration. Then $P\trianglelefteq Q$.
\end{theorem}

See Theorem 11.1 of \cite{dmatm}.

\begin{theorem}
\label{iterability of m-s}
    Suppose $M$ is an $\omega_1+1$-iterable mouse with Woodin cardinal $\delta$ satisfying enough of $ZFC$ and $z\in M\cap \R$. Then the Mitchell-Steel construction of length $\delta$ over $z$ done inside $M$ does not break down. Let $N$ be the output of the construction. Then $N$ is a $z$-mouse of height $\delta$.
\end{theorem}

The proof of Theorem \ref{iterability of m-s} is well known. To show the construction does not break down, by \cite{msbook} it suffices to show universality and solidity of the models $\langle \M_\xi: \xi < \delta \rangle$ built during the construction. \cite{msbook} further reduces this to showing iterability for each $\M_\xi$. An iteration strategy for $M_\xi$ can be defined by lifting iteration trees on $\M_\xi$ to trees on (an initial segment of) $M$ and selecting the branch picked by the strategy for $M$. $\omega_1+1$-iterability of $M$ suffices to obtain the required iterability for each $\M_\xi$. Similarly, $N$ being a $z$-mouse follows from iterability of $M$.

For a premouse $M$ satisfying enough of $ZFC$ and $z\in M\cap \R$, we write $Le[M,z]$ for the output of the Mitchell-Steel construction in $M$ over $z$ (assuming the construction does not break down). $Le[M]$ will refer to $Le[M,\emptyset]$. $Le[M,z]$ is a $z$-premouse. If $M$ is iterable, so is $Le[M,z]$.

We are most interested in cases in which $M$ is a mouse with a Woodin cardinal $\delta$, no largest cardinal, and no total extenders above $\delta$. Then $Le[M|\delta,z]$ is equal to the Mitchell-Steel construction of length $\delta$ over $z$, done inside $M$, and $Le[M,z]$ is an initial segment of $L(Le[M|\delta.z])$.

\begin{remark}
\label{m-s up to inaccessible}
    Suppose $M$ is an $\omega_1+1$-iterable premouse, $z\in M\cap \R$, and $\kappa$ is inaccessible in $M$. Let $\langle \M_\xi : \xi < \kappa\rangle$ be the models of the Mitchell-Steel construction in $M$ of length $\kappa$ over $z$. Suppose an extender is added at step $\xi+1$ in the construction. Let $F^*$, $F$, and $\nu$ be as in Case \ref{active case of m-s construction} of the construction. Then there is $F'\in M|\kappa$ such that $M\models V_{\nu+\omega}\subset Ult(M,F')$ and $F'\cap ([\nu]^{<\omega} \times \M_\xi) = F\upharpoonright\nu$. So we may assume if $F^*$ is used as a background in the construction of length $\kappa$, then $F^*\in M|\kappa$.

    In particular, if $M$ is an $\omega_1+1$-iterable premouse, $z\in M\cap \R$, and $\kappa$ is inaccessible in $M$, then $Le[M|\kappa,z]$ equals the Mitchell-Steel construction of length $\kappa$ over $z$, done in $M$.
\end{remark}
\vspace{2mm}
\subsection{S-constructions}
\label{s-const section}

Below we outline the $S$-construction (this was introduced as the $P$-construction in \cite{self-iter}).

Suppose $M = (J_\gamma^{\vec{E}},\in,\vec{E}\upharpoonright\gamma,E_\gamma,a)$ is a countable $a$-premouse and $\delta\in M$ is a cardinal and cutpoint of $M$. Suppose $ON \cap \bar{S} = \delta + \omega$, $\delta$ is a Woodin cardinal of $\bar{S}$, $\bar{S}$ is definable over $M$, and there is a generic $G$ (for the version of Woodin's extender algebra with $\delta$ propositional letters) such that $\bar{S}[G] = M|\delta+1$. Inductively define a sequence $\langle S_\alpha : \delta + 1 \leq \alpha \leq \gamma \rangle$ as follows. $S_{\delta+1}$ is set to be $\bar{S}$. At a limit $\lambda$, $S_\lambda = \bigcup_{\alpha < \lambda} S_\alpha$. If $M|\lambda$ is active, add a predicate for $E_{\lambda} \cap S_\lambda$ to $S_\lambda$. For the successor step, we define $S_{\alpha + 1}$ by constructing one more level over $S_\alpha$. The construction proceeds until we construct $S_\gamma$, or we reach some $S_\alpha$ such that $\delta$ is not Woodin in $S_\alpha$. We refer to $S_\gamma$ as the maximal $S$-construction in $M$ over $\bar{S}$ if the construction reaches $\gamma$. We are primarily interested in cases where $\delta$ is Woodin in $M$, in which case the construction is guaranteed to reach $\gamma$.

\begin{lemma}
\label{s-const lemma}
    Suppose $M,\bar{S},\delta,\gamma,$ and $G$ are as above. Assume also $M$ is $(\omega_1,\omega_1+1)$-iterable, $\omega$-sound, and $\rho_\omega(M)\geq \delta$. If the construction reaches $\gamma$, then for each $\alpha$ such that $\delta+1\leq \alpha \leq \gamma$, $S_\alpha$ is an $\bar{S}$-mouse and $S_\alpha[G] = M|\alpha$. If also $\alpha < \gamma$, or $\alpha = \gamma$ and $\delta$ is definably Woodin over $S_\alpha$, then $\rho_n(S_\alpha) = \rho_n(M|\alpha)$ for all $n$ and $S_\alpha$ is $\omega$-sound.
\end{lemma}

Lemma 1.5 of \cite{self-iter} gives everything in Lemma \ref{s-const lemma} except the iterability of $S_\gamma$. The iteration strategy for $S_\gamma$ in Lemma \ref{s-const lemma} comes from lifting an iteration tree on $S_\gamma$ to iteration trees on $M$ above $\delta$. In particular, we have:

\begin{fact}
\label{definiability of s-const strat}
    Suppose $M,\bar{S},\delta,\gamma,$ and $G$ are as in Lemma \ref{s-const lemma}. Then the iteration strategy for $S_\gamma$ (as an $\bar{S}$-premouse) is projective in the iteration strategy for $M$ restricted to iteration trees above $\delta$.
\end{fact}

%Suppose $\mathcal{T}$ is an iteration tree on $S_\gamma$. We lift $\T$ to an iteration tree $\mathcal{U}$ on $M$ as follows. At step $\alpha+1$ in the game constructing $\T$, Player I picks an extender $F^\T_\alpha$ from the fine extender sequence of $M^\T_\alpha$ and $F^\T_\alpha$ is applied to $M^\T_\beta$ for some $\beta \leq \alpha$. 

The $S$-construction serves two purposes in what follows. It allows us to ``undo'' generic extensions from Woodin's extender algebra. And combined with the fully-backgrounded Mitchell-Steel construction, it provides an inner model of a premouse with convenient properties.

\begin{definition}
    Let $M$ be an $\omega_1+1$-iterable premouse with a Woodin cardinal and $z\in M \cap \R$. Let $\bar{S}$ be the result of constructing one level of the $\mathcal{J}$-hierarchy over $Le[M|\delta_M,z]$. Let $StrLe[M,z]$ denote the maximal $S$-construction in $M$ over $\bar{S}$.
\end{definition}
\vspace{2mm}
\subsection{Suitable Mice}
\label{suitable mice section}

We now review some results from the core model induction. Most of the concepts below are from \cite{cmi}, with some minor additions. We need to work with mice with an inaccessible cardinal above a Woodin, so in Definition \ref{ss definition} we introduce a modification of the standard notion of a suitable premouse. \cite{cmi} proves the existence of terms in suitable mice capturing certain sets of reals. We will need analogous lemmas for our modified definition. In fact we require more than is stated in \cite{cmi} --- it is essential for our purposes that there is a canonical term capturing each set. Fortunately, this stronger claim is already implicit in the proofs of \cite{cmi}.

For the remainder of this section, we will asssume $ZF + AD + DC + V=L(\R)$ and fix a boldface inductive-like pointclass $\bm{\Gamma}$ such that $\bm{\Gamma}\neq \bm{\Sigma^2_1}$. We then have $\bm{\Gamma} = \bm{\Sigma_1}(J_{\alpha_0}(\R))$ for some $\alpha_0$ beginning an admissible $\Sigma_1$-gap $[\alpha_0,\beta_0]$. Fix a lightface pointclass $\Gamma$ as in Remark \ref{form of ind-like} such that $\bm{\Gamma}$ is the closure of $\Gamma$ under preimages by continuous functions.

\begin{definition}
    Suppose $x\in HC$. Say an $x$-premouse $N$ is $\Gamma$-suitable if $N$ is countable and
    \begin{enumerate}
        \item $N \models$ there is exactly one Woodin cardinal $\delta_N$.
        \item Letting $N_0 = Lp^\Gamma(N|\delta_N)$ and $N_{i+1} = Lp^\Gamma(N_i)$, we have that $N = \bigcup_{i<\omega} N_i$.
        \item If $\xi < \delta_N$, then $Lp^\Gamma(N|\xi) \models \xi$ is not Woodin.
    \end{enumerate}
\end{definition}

\begin{definition}
\label{ss definition}
    Suppose $x\in HC$. Say an $x$-premouse $N$ is $\Gamma$-super-suitable ($\Gamma$-ss) if $N$ is countable and
    \begin{enumerate}
        \item $N \models$ There is exactly one Woodin cardinal $\delta_N$.
        \item $N \models$ There is exactly one inaccessible cardinal above $\delta_N$. We denote this inaccessible by $\nu_N$.
        \item Letting $N_0 = Lp^\Gamma(N|\nu_N)$ and $N_{i+1} = Lp^\Gamma(N_i)$, we have that $N = \bigcup_{i<\omega} N_i$.
        \item For each $\xi \geq \delta_N$, $N|(\xi^+)^N = Lp^\Gamma(N|\xi)$. 
        \item If $\xi < \delta_N$, then $Lp^\Gamma(N|\xi) \models \xi$ is not Woodin.
    \end{enumerate}
\end{definition}

\begin{definition}
    Let $N$ be a mouse and $\delta\in N$. We say $\delta$ is a $\Gamma$-Woodin of $N$ if $\delta$ is Woodin in $Lp^\Gamma(N|\delta)$. %We say $M$ does not reach a $\Gamma$-Woodin if there is no $\Gamma$-Woodin in $M$.
\end{definition}

A $\Gamma$-suitable premouse is a minimal premouse with a $\Gamma$-Woodin cardinal which is closed under $Lp^{\Gamma}$, in that none of its initial segments have this property. Similarly, a $\Gamma$-ss premouse can be considered a minimal premouse with a $\Gamma$-Woodin which is closed under $Lp^\Gamma$ and has an inaccessible cardinal above its $\Gamma$-Woodin. The existence of a $\Gamma$-suitable (or $\Gamma$-ss) premouse is not any stronger than the existence of a premouse with a $\Gamma$-Woodin cardinal. For suppose $N = Lp^{\Gamma}(N|\delta)$, $\delta$ is Woodin in $N$, and no $\xi < \delta$ is a $\Gamma$-Woodin of $N$. If $Q \triangleright N$, $\rho(Q) \leq \delta$, $Q$ is $\delta$-sound, and a set definable over $Q$ witnesses that $\delta$ is not Woodin, then $Q$ must have a $\Gamma$-Woodin cardinal above $\delta$. Otherwise, $Q$ would be iterable by $Q$-structures in $\bm{\Delta_\Gamma}$ and hence in $Lp^\Gamma(N|\delta)$. Then we may build a $\Gamma$-suitable ($\Gamma$-ss) premouse by closing under $Lp^\Gamma$ as many times as is necessary, since this will never construct a $\Gamma$-Woodin.

\begin{definition}
    Let $A \subseteq \R$, $N$ a countable premouse, $\eta$ an uncountable cardinal of $N$, and $\tau\in N^{Col(\omega,\eta)}$. We say that $\tau$ weakly captures $A$ over $N$ if whenever $g$ is $Col(\omega,\eta)$-generic over $N$, $\tau[g] = A \cap N[g]$.
\end{definition}

\begin{lemma}
\label{terms condense}
    Suppose $\mathcal{B}$ is a self-justifying system and $N$ and $M$ are transitive models of enough of ZFC such that $N\in M$. Let $\mathcal{C}$ be a comeager set of $Col(\omega,N)$ generics over $M$ and suppose for each $B\in \mathcal{B}$ there is a term $\tau_B\in M$ such that if $g\in\mathcal{C}$, then $\tau_B[g] = B \cap M[g]$. Let $\pi: \bar{M} \to M$ be elementary with $\{N\} \cup \{\tau_B: B \in \mathcal{B}\}\subset ran(\pi)$. Let $(N,\tau_B) = \pi(\bar{N},\bar{\tau}_B)$. Then whenever $g$ is $Col(\omega,\bar{N})$-generic over $\bar{M}$, $\bar{\tau}_B[g] = B \cap \bar{M}[g]$.
\end{lemma}

See Lemma 3.7.2 of \cite{cmi}.

Let $\beta'$ be the least ordinal greater than $\alpha_0$ such that there is a scale for a universal $\bm{\Pi_1}(J_{\alpha_0}(\R))$ set definable over $J_{\beta'}(\R)$. By Theorem \ref{where sjs appears}, $\beta'= \beta_0$ or $\beta'=\beta_0+1$ and there is a self-justifying system $\G = \{ G_n : n\in\omega\}$ such that 
\begin{align*}
    G_0 = \{(x,y):\,  & x \text{ codes some transitive set } a \text{ and } y \text{ codes an } \omega\text{-sound } \\
    & a\text{-premouse } R \text{ such that } R \text{ projects to } a \text{ and } R \text{ has an } \\
    & \omega_1\text{-iteration strategy in } \bm{\Delta}\},
\end{align*}
and $\G$ is contained in $OD^{<\beta'}(z)$ for some $z\in\mathbb{R}$. Note $G_0\in\bm{\Gamma}$, by part \ref{Gamma is (Sigma^2_1)^Delta} of Remark \ref{form of ind-like}. In fact, $G_0$ is a universal $\bm{\Gamma}$-set (we will not need this property specifically for $G_0$, but we will use that $\G$ contains some universal $\bm{\Gamma}$-set). For ease of notation, assume $\G\subset OD^{<\beta'}$.

%\footnote{Note $G_0 \in \Gamma$, since $y\in C_\Gamma(x)$ if and only if there exists an $\omega$-sound premouse $R$ projecting to $\omega$ such that $y\in R$ and a strategy $\Sigma$ for $R$ in $\Gamma \cap \Gamma^c$. Since $\Gamma = (\Sigma^2_1)^{\Delta_\Gamma}$, this implies $G_0 \in \Gamma$ (see p. 297 of \cite{hacm}).}

\begin{definition}
\label{definition of standard terms}
    Suppose $B \subset \mathbb{R}$, $N$ is a premouse, and $\eta$ is a cardinal of $N$. Let $\tau^N_{B,\eta}$ be the set of pairs $(\sigma,p)\in N$ such that
    \begin{enumerate}
        \item $\sigma$ is a  $Col(\omega,\eta)$-standard term for a real,
        \item $p\in Col(\omega,\eta)$, and
        \item for comeager many $g\subset Col(\omega,\eta)$ which are $Col(\omega,\eta)$-generic over $N$ such that $p\in g$, $\sigma[g]\in B$.
    \end{enumerate}
    For $n\in\omega$, let $\tau^N_{n,\eta} = \tau^N_{G_n,\eta}$ and if $N$ has a Woodin cardinal let $\tau^N_n = \tau^N_{n,\delta_N}$.
\end{definition}

\begin{lemma}
\label{terms exist}
    Suppose $N$ is a $\Gamma$-suitable or $\Gamma$-ss premouse, $z\in N$, $B\in OD^{<\beta'}(z)$, and $\eta$ is a cardinal of $N$. Then $\tau^N_{B,\eta}$ is in $N$.
\end{lemma}

See the proof of Lemma 3.7.5 of \cite{cmi}. In Lemma 5.4.3 of \cite{cmi}, Lemma \ref{terms condense} is used to show:

\begin{lemma}[Woodin]
\label{term relation condensation}
    Suppose $z\in\R$, $N$ is a $\Gamma$-suitable (or $\Gamma$-ss) $z$-premouse, and $\mathcal{B}$ is a sjs containing a universal $\bm{\Sigma_1}(J_{\alpha_0}(\R))$-set such that each $B\in\mathcal{B}$ is $OD^{<\beta'}(z)$. Suppose $\pi:M\to N$ is $\Sigma_1$-elementary and for every $B\in\mathcal{B}$ and $\eta \geq \delta_N$, $\tau^N_{B,\eta}\in range(\pi)$. Then
    \begin{enumerate}
        \item $M$ is $\Gamma$-suitable ($\Gamma$-ss) and
        \item $\pi(\tau^M_{B,\bar{\eta}}) = \tau^N_{B,\eta}$, where $\bar{\eta}$ is such that $\pi(\bar{\eta}) = \eta$.
    \end{enumerate}
\end{lemma}

As a result of Lemmas \ref{terms condense} and \ref{terms exist} we have:

\begin{corollary}
\label{standard terms capture}
    If $N$ is $\Gamma$-suitable or $\Gamma$-ss and $\eta$ is an uncountable cardinal of $N$, then $\tau^N_{n,\eta}$ weakly captures $G_n$.
\end{corollary}

\begin{definition}
    Let $\T$ be a normal iteration tree on a $\Gamma$-suitable (or $\Gamma$-ss) premouse $N$. Suppose also $\T$ is below $\delta_N$. Say $\T$ is $\Gamma$-short if for all limit $\xi \leq lh(\T)$, $Lp^\Gamma(\M(\T\upharpoonright\xi))\models \delta(\T\upharpoonright\xi)$ is not Woodin. Otherwise, say $\T$ is $\Gamma$-maximal.
\end{definition}

\begin{definition}
    Let $N$ be a $\Gamma$-suitable ($\Gamma$-ss) premouse with an $(\omega_1,\omega_1)$-iteration strategy $\Sigma$. Say $\Sigma$ is fullness-preserving if whenever $P$ is an iterate of $N$ by $\Sigma$ via an iteration below $\delta_N$, then
    \begin{enumerate}
        \item if the branch to $P$ does not drop, then $P$ is $\Gamma$-suitable ($\Gamma$-ss), and
        \item if the branch to $P$ does drop, then $P$ has an $\omega_1$-iteration strategy in $J_{\alpha_0}(\mathbb{R})$.
    \end{enumerate}
\end{definition}

\begin{remark}
\label{no gamma woodin gives delta strat}
    Let $N$ be a $\Gamma$-suitable (or $\Gamma$-ss) mouse with a fullness-\\preserving iteration strategy $\Sigma$. Suppose $P\triangleleft N|\delta_N$, and $\Sigma'$ is the iteration strategy for $P$ given by restricting the domain of $\Sigma$ to trees on $P$. Suppose $\T$ is an iteration tree on $P$ according to $\Sigma'$. Then the branch $b$ through $\T$ chosen by $\Sigma'$ can be determined from $\Q(\T)$. And $\Q(\T)$ is the unique $\M(\T)$-mouse projecting to $\omega$ with an iteration strategy in $\bm{\Delta}$. It follows from Remark \ref{form of ind-like} and the uniqueness of $\Q(\T)$ that $\Sigma'$ is coded by a set in $\bm{\Delta}$.
\end{remark}

\begin{definition}
    Let $\T$ be a $\Gamma$-maximal iteration tree on a $\Gamma$-suitable (or $\Gamma$-ss) premouse $N$ and let $b$ be a cofinal branch through $\T$. Say $b$ respects $\vec{G}_n$ if $i^\T_b(\tau^N_{k,\eta}) = \tau^{M^\T_b}_{k,i_b(\eta)}$ for all $k < n$ and every cardinal $\eta$ of $N$ above $\delta_N$.
\end{definition}

\begin{definition}
     Let $N$ be a $\Gamma$-suitable (or $\Gamma$-ss) mouse with a fullness-\\preserving iteration strategy $\Sigma$. Say $\Sigma$ is guided by $\G$ if whenever $\T$ is an iteration tree according to $\Sigma$ of limit length and $b = \Sigma(\T)$, then
    \begin{enumerate}
        \item if $\T$ is $\Gamma$-short, then $\Q(b,\T)$ exists and $\Q(b,\T) \in Lp^\Gamma(\M(\T))$, and
        \item if $T$ is $\Gamma$-maximal, then $\Sigma(b)$ respects $\vec{G}_n$ for all $n\in\omega$.
    \end{enumerate}
\end{definition}

\begin{lemma}
\label{guided strat not in gamma}
    If $N$ is $\Gamma$-suitable (or $\Gamma$-ss) and $\Sigma$ is an iteration strategy for $N$ which is guided by $\mathcal{G}$, then $\Sigma$ is not in $\bm{\Gamma}$.
\end{lemma}
\begin{proof}
    There is $n\in\omega$ such that $G_n$ is a universal $\bm{\Gamma^c}$-set. Then $y\in G_n$ if and only if there exists a countable, complete iterate $N^*$ of $N$ according to $\Sigma$ and $g\in \R$ which is $Col(\omega,\R)$-generic over $N^*$ such that $y\in\tau^{N^*}_n[g]$. Since $\bm{\Gamma}$ is closed under projection, if $\Sigma$ were in $\bm{\Gamma}$, $G_n$ would also be in $\bm{\Gamma}$.
\end{proof}

\begin{theorem}[Woodin]
\label{nice strategy exists}
    For any $x\in HC$, there is a (unique) $\omega$-sound, $\Gamma$-suitable $x$-mouse $W_x$ projecting to $x$ with a (unique) iteration strategy that is fullness-preserving, condenses well,\footnote{In the sense of Definition 5.3.7 of \cite{cmi}.} and is guided by $\G$. Similarly, there is  a (unique) $\omega$-sound, $\Gamma$-ss $x$-mouse $M_x$ projecting to $x$ with a (unique) iteration strategy that is fullness-preserving, condenses well, and is guided by $\G$.
\end{theorem}

Chapter 5 of \cite{cmi} demonstrates the existence of such a $\Gamma$-suitable mouse. It is not difficult to see this gives the existence of the required $\Gamma$-ss mouse as well.

For any $\Gamma$-suitable (or $\Gamma$-ss) premouse $N$ and any $n\in\omega$, let
\begin{align*}
    \gamma^N_n = Hull^N(\{\tau^N_i: i < n\}) \cap \delta_N.
\end{align*}

The regularity of $\delta_N$ in $N$ implies each $\gamma^N_n$ is an ordinal. Lemma \ref{term relation condensation} can be used to show:

\begin{fact}
    $\langle \gamma^N_n : n \in \omega \rangle$ is cofinal in $\delta_N$.
\end{fact}

\begin{lemma}
\label{branch agreement}
    Let $\T$ be a normal iteration tree on a $\Gamma$-suitable (or $\Gamma$-ss) premouse $N$ and let $b$ and $c$ be branches through $\T$ which respect $\vec{G}_n$. Then $i^\T_b\upharpoonright\gamma^N_n = i^\T_c\upharpoonright\gamma^N_n$. Moreover, if $b$ and $c$ both respect $\vec{G}_n$ for all $n$, then $b=c$.
\end{lemma}

See Lemma 6.25 of \cite{hacm}. 

Lemma \ref{branch agreement} implies if $b$ is the branch through $\T$ chosen by the nice iteration strategy for a $\Gamma$-suitable premouse given by Theorem \ref{nice strategy exists} and $c$ is any branch respecting $\vec{G}_n$, then $i^\T_b$ and $i^\T_c$ agree up to $\gamma^N_n$. In particular, to track the iteration of a $\Gamma$-suitable mouse up to some point below its least Woodin, it is sufficient to know finitely many of the sets in $\G$.

Suppose $M$ is a countable premouse with an $(\omega_1,\omega_1+1)$-iteration strategy.\footnote{The suitable mice from Theorem \ref{nice strategy exists} satisfy this. We only explicitly required these to have $(\omega_1,\omega_1)$-iteration strategies, but since $ZF + AD$ implies $\omega_1$ is measurable, an $(\omega_1,\omega_1)$-iteration strategy induces an $(\omega_1,\omega_1+1)$-iteration strategy.} Together, the Comparison Lemma and the Dodd-Jensen Lemma imply the collection of countable, complete iterates of $M$, together with the iteration maps between them, forms a directed system.

\cite{hodbelowtheta} presents work of Steel and Woodin analyzing the direct limit of all countable, complete iterates of $M^\#_\omega$. This direct limit cut to its least Woodin is $(HOD||\Theta)^{L(\R)}$. \cite{hacm} goes further in showing that the entire class $HOD^{L(\R)}$ is a strategy mouse. The iteration maps through trees on $M^\#_\omega$ are approximated using indiscernibles, analogously to the use of terms in Lemma \ref{branch agreement}. These approximations are merged to give an ordinal definable definition of the direct limit in $L(\R)$. In particular, initial segments of the direct limit maps are definable from finitely many indiscernibles.

In place of $M^\#_\omega$, we shall analyze the direct limit of a $\Gamma$-suitable mouse and prove that portions of the direct limit maps are definable within a $\Gamma$-ss mouse.

Our task is simpler in that we only need to reach up to $\bm{\delta_\Gamma^+}$, which we show in Section \ref{length of direct limit section} is below the least Woodin of our direct limit. So a single approximation using only finitely many sets from $\G$ will suffice. Another advantage we have is that there is no harm in working over a real parameter, so we can work in a $\Gamma$-ss mouse over a real which codes $W_0$. On the other hand, we will have some extra work to do in Section \ref{definability section} ensuring enough information about $\Gamma$ and $\G$ is definable in a $\Gamma$-ss mouse before we internalize the directed system in Section \ref{internalization section}.

\cite{hacm} also makes use of the fact that the derived model of $M^\#_\omega$ is essentially $L(\R)$. So for $x\in M^\#_\omega \cap \R$, a $\Sigma^2_1$ statement about $x$ is true if and only if it holds in the derived model of $M^\#_\omega$. In particular, there is a natural way to ask about $\Sigma^2_1$ truth inside of $M^\#_\omega$. A second, though minor, inconvenience of having to use a $\Gamma$-suitable mouse is we cannot talk about its derived model, since it only has one Woodin. Instead we will use the fine-structural witness condition of \cite{cmi}.

\begin{remark}
    We can associate to any $\Sigma_1$-formula $\phi$ a sequence of formulas $\langle \phi^k:k<\omega\rangle$ such that for any ordinal $\gamma$ and any real $z$, $J_{\gamma+1}(\R)\models \phi[z] \iff (\exists k) J_\gamma(\R) \models \phi^k[z]$. Moreover, the map $\phi \to \langle \phi^k:k<\omega\rangle$ is recursive.
\end{remark}

\begin{definition}
    Suppose $\phi(v)$ is a $\Sigma_1$-formula and $z\in\R$. A $\langle \phi,z\rangle$-witness is an $\omega$-sound $z$-mouse $N$ in which there are $\delta_0 < ... < \delta_9$, $\mathcal{S}$, and $\T$ such that N satisfies the formulae expressing
    \begin{enumerate}
        \item ZFC,
        \item $\delta_0 < ... < \delta_9$ are Woodin,
        \item $\mathcal{S}$ and $\T$ are trees on some $\omega \times \eta$ which are absolutely complementing in $V^{Col(\omega,\delta_9)}$, and
        \item For some $k < \omega$, $\rho[T]$ is the $\Sigma_{k+3}$-theory (in the language with names for each real) of $J_\gamma(\R)$, where $\gamma$ is least such that $J_\gamma(\R) \models \phi^k[z]$.
    \end{enumerate}
\end{definition}

Other than iterability, the rest of the properties of being a $\langle \phi,z\rangle$-witness are first order. The following two lemmas illustrate the usefulness of this definition.

\begin{lemma}
\label{witness gives truth}
    If there is a $\langle \phi,z\rangle$-witness, then $L(\R)\models \phi[z]$.
\end{lemma}

\begin{lemma}
\label{truth gives witness}
    Suppose $\phi$ is a $\Sigma_1$-formula, $z\in\R$, $\gamma$ is a limit ordinal, and $J_\gamma(\R)\models \phi[z]$. Then there is a $\langle\phi,z\rangle$-witness $N$ such that the iteration strategy for $N$ restricted to countable trees is in $J_\gamma(\R)$. By taking a Skolem hull, we can also ensure $\rho_\omega(N) = \omega$.
\end{lemma}
\vspace{2mm}
\section{The Inductive-Like Case}
\label{ind-like chapter}

In this section we will prove Theorem \ref{main thm}. We now assume $ZF + AD + DC + V=L(\R)$ and fix a boldface inductive-like pointclass $\bm{\Gamma}$. By a reflection argument, we may assume $\bm{\Gamma} \neq \bm{\Sigma^2_1}$.\footnote{Suppose the theorem fails for $\bm{\Gamma} = \bm{\Sigma^2_1}$. Then an initial segment of $L(\R)$ satisfying a large fragment of $ZF + AD + DC$ satisfies this. Reflecting this gives an inital segment $N$ of $L(\R)$ below $\bm{\delta^2_1}$ such that $\bm{\Gamma}' = (\bm{\Sigma^2_1})^N$ is an inductive-like pointclass in $L(\R)$ and $N$ satisfies that there exists a sequence of distinct $\bm{\Gamma}'$ sets of length $\bm{\delta_\Gamma}^+$. Since $N$ satisfies enough of $ZF + AD + DC + V=L(\R)$, the proof that follows will give a contradiction in $N$. In this case, the iteration strategy for the $\mathbf{\Gamma}'$-suitable mouse used in the proof will not be in $N$. But this does not effect the argument --- it is enough that the strategy is in $L(\R)$.}

Let $\bm{\Delta} = \bm{\Delta_{\Gamma}}$ and let $[\alpha_0,\beta_0],\beta'$, $\Gamma$, and $\G$ be as in Section \ref{suitable mice section}. We will also refer to the mouse operators $x\to W_x$ and $x\to M_x$ from Theorem \ref{nice strategy exists} and use the notation for standard terms from Definition \ref{definition of standard terms}.

In Sections \ref{length of direct limit section} through \ref{internalization section} we analyze the directed system of iterates of a suitable mouse and show the directed system can be approximated inside a larger suitable mouse. Section \ref{strle lemmas section} covers some lemmas about the StrLe construction inside a suitable mouse. Section \ref{reflection section} contains a lemma we will use to obtain witnesses for $\Sigma_1$ statements inside an initial segment of a suitable mouse. Finally, Theorem \ref{main thm} is proven in Section \ref{main thm section}.

One of the key ideas to our proof of Theorem $\ref{main thm}$ is a different coding than the one used in \cite{twoapp} and \cite{hra}. In \cite{hra}, $\bm{\Sigma^1_{2n+2}}$ sets are coded by conditions in the extender algebra at the least Woodin of some complete iterate $N$ of $M^\#_{2n+1}$. The reflection argument from \cite{twoapp} ensures a code for each $\bm{\Sigma^1_{2n+2}}$ set appears below the least $<\delta_N$-strong cardinal $\kappa_N$ of some iterate $N$ (in fact it gives a uniform bound below $\kappa_N$). But this reflection argument depends upon the pointclass $\bm{\Sigma^1_{2n+2}}$ not being closed under coprojection.

Our proof of Theorem \ref{main thm} instead codes $\bm{\Gamma}$-sets by sets of conditions in the extender algebra of some $\Gamma$-suitable mouse $N$. A weaker reflection argument than the one in \cite{twoapp} is used to contain each code in $N|\kappa_N$. This weaker reflection is sufficient for the proof.\\

\subsection{The Direct Limit}
\label{length of direct limit section}

Let $W = W_0$ and let $\mathcal{I}$ be the directed system of countable, complete iterates of $W$ according to its $(\omega_1,\omega_1)$-iteration strategy. Let $M_\infty$ be the direct limit of $\mathcal{I}$. For $M,N\in \mathcal{I}$ and $N$ an iterate of $M$, let $\pi_{M,N}: M \to N$ be the iteration map and $\pi_{M,\infty}: M \to M_\infty$ the direct limit map. Here we demonstrate a few properties of $M_\infty$. The proofs of this section are generalizations of arguments in \cite{ooimt} and \cite{hacm} giving analogous properties of the direct limit of all countable, complete iterates of $M^\#_\omega$.

\begin{lemma}
\label{pwo is long}
    $\kappa_{M_\infty} \leq \bm{\delta_\Gamma}$ \footnote{In fact $\kappa_{M_\infty} = \bm{\delta_\Gamma}$, but we don't need this.}
\end{lemma}
\begin{proof}    
    Suppose $\xi < \kappa_{M_\infty}$. Let $M\in \mathcal{I}$ and $\bar{\xi}\in M$ be such that $\pi_{M,\infty}(\bar{\xi}) = \xi$. Let $P$ be an initial segment of $M$ such that $\bar{\xi}\in P$ and the largest cardinal of $P$ is both a cutpoint and a cardinal of $M$. The iteration strategy $\Sigma$ for $P$ is in $\bm{\Delta}$ by Remark \ref{no gamma woodin gives delta strat}. Let $\I_P$ be the directed system of countable, complete iterates of $P$ by $\Sigma$. Then $\bar{\xi}$ is sent to $\xi$ by the direct limit map of this system, since the largest cardinal of $P$ is a cutpoint and a cardinal of $M$. So a prewellordering of height $\xi$ is projective in $\Sigma$ and therefore $\bm{\delta_\Gamma} > \xi$.
\end{proof}

\begin{lemma}
\label{pwo is short}
    $\delta_{M_\infty} > (\bm{\delta_\Gamma})^+$
\end{lemma}
\begin{proof}    
    Let $\Sigma$ be the $(\omega_1,\omega_1)$-iteration strategy for $W$. Recall $\Sigma$ is not in $\bm{\Gamma}$. We will show $\Sigma$ is in $S(\delta_{M_\infty})\backslash S(\bm{\delta_\Gamma}^+)$.
    
    \begin{claim}
        $\Sigma$ is $\delta_{M_\infty}$-Suslin.
    \end{claim}
    \begin{proof}
        Let $\T$ be a tree on $(\omega \times \omega) \times \delta_{M_\infty}$ such that $(x,y,f)\in [\T]$ if and only if $x$ codes a countable iteration tree $\mathcal{S}$ on $W$ of limit length, $y$ codes a cofinal, wellfounded branch $b$ through $\mathcal{S}$, and $f$ codes an embedding $\pi: M_b^{\mathcal{S}} \to M_\infty$ such that $\pi \circ i^{\mathcal{S}}_b = \pi_{W,\infty}$. Let $\Sigma' = \rho[\T]$.

        If $(x,y)\in \Sigma$, then $x$ codes an iteration tree $\mathcal{S}$ on $W$ according to $\Sigma$ and $y$ codes the cofinal, wellfounded branch $b$ through $\mathcal{S}$ chosen by $\Sigma$. And $\pi_{M^\mathcal{S}_b,\infty}\circ i^\mathcal{S}_b = \pi_{W,\infty}$. So if $f:\omega\to \delta_{M_\infty}$ codes the embedding $\pi_{M^\mathcal{S}_b,\infty}$, then $(x,y,f)\in[\T]$. Thus $(x,y)\in\Sigma'$.

        On the other hand, suppose $(x,y) \in \Sigma'$ and $x$ codes an iteration tree $\mathcal{S}$ according to $\Sigma$. Fix $f:\omega\to \delta_{M_\infty}$ such that $(x,y,f)\in[\T]$. Let $b$ be the branch coded by $y$ and $\pi$ the embedding coded by $f$.
        \begin{subclaim}
            For all $n$, $\pi^\mathcal{S}_b(\tau^W_n) = \tau^{M^\mathcal{S}_b}_n$.
        \end{subclaim}
        \begin{proof}
            Let $Q \in \mathcal{I}$ be such that $range(\pi)\subseteq range(\pi_{Q,\infty})$. Let $\pi' = \pi_{Q,\infty}^{-1} \circ \pi$. Then $\pi':M^{\mathcal{S}}_b \to Q$ and $\pi'(i^{\mathcal{S}}_b(\tau^W_n)) = \tau^Q_n$. Then by Lemma \ref{term relation condensation}, $i^{\mathcal{S}}_b(\tau^W_n) = \tau^{M_b^{\mathcal{S}}}_n$.
        \end{proof}

        From the subclaim and the last part of Lemma \ref{branch agreement}, we have that $(x,y)\in\Sigma$.

        We can now characterize $\Sigma$ as the set of $(x,y)\in\R\times\R$ such that 
        \begin{enumerate}
            \item $x$ codes an iteration tree $\mathcal{S}$ on $W$ of limit length,
            \item $y$ codes a cofinal, wellfounded branch through $\mathcal{S}$,
            \item $(x,y)\in\Sigma'$, and
            \item \label{initial segments of tree are good} for any $(x_0,y_0) \leq_T x$ such that $x_0$ codes a proper initial segment $\mathcal{S}_0$ of $\mathcal{S}$ of limit length and $y_0$ codes the branch through $\mathcal{S}_0$ determined by $\mathcal{S}$, $(x_0,y_0)\in \Sigma'$.
        \end{enumerate}

        Condition \ref{initial segments of tree are good} is just to guarantee $\mathcal{S}$ is in the domain of $\Sigma$. It does so because any proper initial segment $\mathcal{S}_0$ of $\mathcal{S}$ is coded by some real computable from $x$. From this, and the preceding paragraphs, it is clear these conditions characterize $\Sigma$. Since $\Sigma'$ is $\delta_{M_\infty}$-Suslin, this characterization of $\Sigma$ makes plain that $\Sigma$ is also $\delta_{M_\infty}$-Suslin.
    \end{proof}
    
    \begin{claim}
        $\bm{\Gamma} = S(\bm{\delta_\Gamma})$.
    \end{claim}
    \begin{proof}
        First, let's establish $\bm{\Gamma}$ is Suslin (we say a pointclass is Suslin if it equals $S(\lambda)$ for some cardinal $\lambda$). Let \begin{align*}
            \Omega= \{\bm{\Sigma_1}(J_\gamma(\mathbb{R})) : \gamma < \alpha_0 \text{ and } \gamma \text{ begins a } \Sigma_1\text{-gap}\}.
        \end{align*}
        It follows from Theorem \ref{extended projective hierarchy to levy hierarchy} that $\bm{\Gamma}$ is the minimal non-selfdual pointclass closed under projection which contains every pointclass in $\Omega$. Let 
        \begin{align*}
            \Psi = \{\bm{\Sigma_1}(J_\gamma(\R)) \in \Omega : \, \bm{\Sigma_1}(J_\gamma(\R)) \text{ is Suslin}\}.
        \end{align*}
        By Theorem \ref{classification of suslin pointclasses}, $\Psi$ is cofinal in $\Omega$. But the minimal Suslin pointclass larger than any element of $\Psi$ is just the minimal non-selfdual pointclass closed under projection which contains every pointclass in $\Omega$ (by part \ref{ind-like case of suslin classification} of Theorem \ref{classification of suslin pointclasses}). Since $\Psi$ is cofinal in $\Omega$, this is $\bm{\Gamma}$.
    
        So $\bm{\Gamma} = S(\lambda)$ for some cardinal $\lambda$. By the Kunen-Martin Theorem, there is a prewellordering of length $\lambda$ in $\bm{\Gamma}$ but no prewellordering of length $\lambda^+$. The latter implies that $\lambda\geq\bm{\delta_\Gamma}$, since $\bm{\delta_\Gamma}$ is a limit cardinal,\footnote{See Theorem 7D.8 of \cite{mosdst}.} and since there are prewellorderings of length $\alpha$ in $\bm{\Gamma}$ for all $\alpha<\bm{\delta_\Gamma}$. The former implies that $\lambda\leq\bm{\delta_\Gamma}$, since there is no prewellordering of length $\bm{\delta_\Gamma^+}$ in $\bm{\Gamma}$ (Otherwise a proper initial segment of this prewellordering would be of length $\bm{\delta_\Gamma}$, giving a prewellordering of length $\bm{\delta_\Gamma}$ in $\bm{\Delta}$). So $\bm{\delta_\Gamma}=\lambda$.
    \end{proof}
    
    By the previous two claims, $\Sigma \in S(\delta_{M_\infty}) \backslash S(\bm{\delta_\Gamma})$. In particular, $\delta_{M_\infty} \geq \lambda'$ where $\lambda'$ is the next Suslin cardinal after $\bm{\delta_\Gamma}$.\footnote{In fact $\delta_{M_\infty} = \lambda'$, but we don't need this.} But $cof(\lambda') = \omega$ by part \ref{ind-like case of suslin classification} of Theorem \ref{classification of suslin pointclasses}, so $\delta_{M_\infty} \geq \lambda' > \bm{\delta^+_\Gamma}$.
\end{proof}

\begin{lemma}
\label{measurable or cof omega}
    Suppose $\mu < \delta_{M_\infty}$ is a regular cardinal of $M_\infty$. Then $\mu$ is not measurable in $M_\infty$ if and only if $\mu$ has cofinality $\omega$ in $L(\R)$.
\end{lemma}
\begin{proof}
    Suppose $\mu$ is not measurable in $M_\infty$. Fix $M \in \mathcal{I}$ and $\bar{\mu}$ such that $\pi_{M,\infty}(\bar{\mu}) = \mu$. Then $\bar{\mu}$ is regular but not measurable in $M$. Since $M$ is countable, there is a sequence of ordinals $\langle \bar{\xi}_n: n < \omega \rangle$ cofinal in $\bar{\mu}$. Let $\xi_n = \pi_{M,\infty}(\bar{\xi}_n)$. Since $\bar{\mu}$ is regular and not measurable in $M$, $\pi_{M,\infty}$ is continuous at $\bar{\mu}$ (This is because $\pi_{M,\infty}$ is essentially an iteration embedding --- in fact it is an iteration embedding in $V^{Col(\omega,\R)}$. And any iteration embedding is continuous at a cardinal which is regular but not measurable, since ultrapower embeddings are continuous at such cardinals). So $\langle \xi_n : n < \omega\rangle $ is cofinal in $\mu$.

    Now suppose $\mu$ has cofinality $\omega$ in $L(\R)$. Let $\langle \xi_n: n <\omega \rangle$ be cofinal in $\mu$. Fix $M \in \mathcal{I}$ such that there is $\bar{\mu}\in M$ and $\langle \bar{\xi}_n : n < \omega\rangle \subset M$ with $\pi_{M,\infty}(\bar{\mu}) = \mu$ and $\pi_{M,\infty}(\bar{\xi}_n) = \xi_n$. If $\mu$ is measurable in $M_\infty$, then there is a total extender $F$ on the fine extender sequence of $M$ with critical point $\bar{\mu}$. Let $M'$ be the ultrapower of $M$ by $F$ and $j:M \to M'$ the embedding induced by $F$. Then for any $n<\omega$,
    \begin{align*}
        \xi_n &= \pi_{M,\infty}(\bar{\xi}_n)\\
        &= \pi_{M',\infty}\circ j(\bar{\xi}_n)\\
        &= \pi_{M',\infty}(\bar{\xi}_n)\\
        &< \pi_{M',\infty}(\bar{\mu})\\
        &<  \pi_{M',\infty}\circ j(\bar{\mu})\\
        &= \mu.
    \end{align*}
    So $\pi_{M',\infty}(\bar{\mu})$ is an upper bound for $\bar{\xi}_n$ below $\mu$, a contradiction.
\end{proof}
\vspace{2mm}
\subsection{Definability in Suitable Mice}
\label{definability section}

\begin{lemma}
\label{Lp is definable}
 Suppose $N$ is a premouse satisfying enough of $ZFC$, $\nu$ is a cardinal of $N$, $Lp^\Gamma(a) \subset N$ for each $a\in N|\nu$, and $\tau \in N^{Col(\omega,\nu)}$ weakly captures $G_0$. Then the map with domain $N|\nu$ defined by $a \mapsto Lp^\Gamma(a)$ is definable in $N$ from $\tau$.
\end{lemma}
\begin{proof}
    Recall
    \begin{align*}
        G_0 = \{(x,y):\,  & x \text{ codes some transitive set } a \text{ and } y \text{ codes an } \omega\text{-sound } \\
        & a\text{-premouse } R \text{ such that } R \text{ projects to } a \text{ and } R \text{ has an } \\
        & \omega_1\text{-iteration strategy in } \bm{\Delta}\},
    \end{align*}
    Fix $a\in N|\nu$. If $R$ is any set in $N|\nu$ and $g$ is any $Col(\omega,\nu)$-generic over $N$, then there are reals $x$ and $y$ in $N[g]$ coding $a$ and $R$, respectively. It is easy to see from this that $Lp^\Gamma(a)$ is
    \begin{align*}
        \bigcup\{R\in N : \emptyset \Vdash^N_{Col(\omega,\nu)} (\exists x,y)[ (x,y)\in\tau \wedge x \text{ codes } a \wedge y \text{ codes } R]\}.
    \end{align*}
\end{proof}

\begin{corollary}
\label{lower part is definable}
    If $P$ is $\Gamma$-ss, then the map with domain $P|\nu_P$ defined by $a \mapsto Lp^\Gamma(a)$ is definable in $P$ from $\tau^P_{0,\nu_P}$.
\end{corollary}
\begin{proof}
    It is clear from Remark \ref{lp(a) contained in lp(b)} and Corollary \ref{standard terms capture} that $P$ and $\tau^P_{0,\nu_P}$ satisfy the conditions of Lemma \ref{Lp is definable}.
\end{proof}

\begin{lemma}
\label{terms definable}
    Suppose $P$ is $\Gamma$-ss and $N \in P|\nu_P$ is $\Gamma$-suitable. Then 
    \newline $\{\tau^N_{n,\mu}: \mu \text{ is an uncountable cardinal of } N \}$ is definable in $P$ from $N$ and $\tau^P_{n,\nu_P}$ (uniformly in $P$ and $N$).
\end{lemma}
\begin{proof}
    Let $\mu$ be an uncountable cardinal of $N$.

    Note if $g$ is $Col(\omega,\nu_P)$-generic over $P$ and $f\in P$ is a surjection of $\nu_P$ onto $\mu$, then $f\circ g$ is $P$-generic for $Col(\omega,\mu)$. In particular, $f \circ g$ is $N$-generic for $Col(\omega,\mu)$. Fix such an $f$ which is minimal in the constructibility order of $P$. Let
    \begin{align*}
        \tau_{n,\mu} = \,&\{(\sigma,p): \sigma \text{ is a } Col(\omega,\mu) \text{-standard term for a real}, p\in Col(\omega,\mu),\\ 
        &\text{ and } \emptyset \Vdash^P_{Col(\omega,\nu_P)} (\check{p}\in \check{f}\circ \dot{g} \rightarrow \check{\sigma}[\check{f} \circ \dot{g}] \in \tau^P_{n,\nu_P})\}
    \end{align*}
    
    It is clear that $\tau_{n,\mu}$ is definable in $P$ from $N$, $\mu$, and $\tau^P_{n,\nu_P}$. It suffices to show $\tau_{n,\mu} = \tau^N_{n,\mu}$.
    
    $\tau_{n,\mu} \subseteq \tau^N_{n,\mu}$ by Definition \ref{definition of standard terms} and that comeager many $h\subset Col(\omega,\mu)$ which are generic over $N$ are of the form $f \circ g$ for some $g$ which is $Col(\omega,\nu_P)$-generic over $P$.
    
    On the other hand, suppose $(\sigma,p) \in \tau^N_{n,\mu}$. By Corollary \ref{standard terms capture}, $\sigma[h]\in G_n$ for any $h$ which is $Col(\omega,\mu)$-generic over $N$ such that $p\in h$. In particular, $\sigma[f \circ g] \in \tau^P_{n,\nu_P}[g]$ for any $g$ which is $Col(\omega,\nu_P)$-generic over $P$ such that $p\in f\circ g$. Thus $(\sigma,p) \in \tau_{n,\mu}$.
\end{proof}

We will also need versions of Corollary \ref{lower part is definable} and Lemma \ref{terms definable} in generic extensions of $\Gamma$-ss mice.

\begin{lemma}
\label{term works in generic extension}
    Suppose $B\subseteq \mathbb{R}$, $P$ is a premouse, $\delta$ is Woodin in $P$, $\mu \geq \delta$, $\tau\in P^{Col(\omega,\mu)}$ weakly captures $B$ over $P$, and $y$ is $Ea_P$-generic over $P$. Then there is $\tau'\in P[y]^{Col(\omega,\mu)}$ which weakly captures $B$ over $P[y]$. Moreover, $\tau'$ is definable in $P[y]$ from $\tau$ and $y$ (uniformly).
\end{lemma}
\begin{proof}

    $Col(\omega,\mu)$ is universal for pointclasses of size $\mu$. So there is a complete embedding $\Phi: Ea_p \times Col(\omega,\mu) \to Col(\omega,\mu)$.\footnote{In the sense of Definition 7.1 of Chapter 7 of \cite{kunen}.} If $g$ is $Col(\omega,\mu)$-generic over $P$, let $(y_g,f_g)$ be the $Ea_P \times Col(\omega,\mu)$-generic consisting of all conditions $(p,q)\in Ea_P \times Col(\omega,\mu)$ such that $\Phi((p,q))\in g$ (see Chapter 7, Theorem 7.5 of \cite{kunen}). Let
    \begin{align*}
        \tau^* = \{(\sigma,(p,q)): & \,\sigma \text{ is an } Ea_P\text{-term for a } Col(\omega,\mu)\text{-standard term }\\
        & \text{for a real, } (p,q)\in Ea_P\times Col(\omega,\mu), \text{ and }\\
        & \Phi((p,q)) \Vdash^P_{Col(\omega,\mu)} \sigma[y_g][f_g] \in \tau[g]\}.
    \end{align*}

    \begin{claim}
    \label{claim for lemma on term definability}
        For any $(y,f)$ which is $Ea_P\times Col(\omega,\mu)$-generic over $P$,\\ $\tau^*[y][f] = B \cap P[y][f]$.
    \end{claim}
    \begin{proof}
        Suppose $x\in\tau^*[y][f]$. $x = \sigma[y][f]$ for some $(\sigma,(p,q))\in\tau^*$ such that $p\in y$ and $q\in f$. Let $g$ be $Col(\omega,\mu)$-generic such that $y_g = y$ and $f_g = f$. In particular, $\Phi((p,q))\in g$. Then $P[g]\models \sigma[y_g][f_g]\in \tau[g]$. Since $x= \sigma[y_g][f_g]$ and $\tau[g]= B \cap P[g]$, $x\in B \cap P[g]$.

        Now suppose $x\in B\cap P[y][f]$. Let $\sigma$ be an $Ea_P$-term for a $Col(\omega,\mu)$-standard term for a real such that $x= \sigma[y][f]$.

        $\bigcup \Phi''\{(p,q):(p,q)\in y\times g\}$ is a function $g_1:S\to\mu$ for some $S\subseteq \omega$. Let
        \begin{align*}
            \mathbb{Q} = \{r\in Col(\omega,\mu): domain(r)\cap S = \emptyset\}
        \end{align*}
        ($\mathbb{Q}$ is the quotient of $Col(\omega,\mu)$ by $g_1$). Let $g_2$ be $\mathbb{Q}$-generic over $P[g_1]$. Then $g = g_1 \cup g_2$ is $Col(\omega,\mu)$-generic over $P$.

        We have $x\in \tau[g]$. Pick $s\in g$ such that $s \Vdash^P_{Col(\omega,\mu)} \sigma[y_g][f_g]\in \tau[g]$. $s = r_1 \cup r_2$ for some $r_1\in g_1$ and $r_2\in g_2$.

        \begin{subclaim}
            $r_1 \Vdash^P_{Col(\omega,\mu)} \sigma[y_g][f_g]\in \tau$.
        \end{subclaim}
        \begin{proof}
            Suppose not. Then there is $g'_2$ which is $\mathbb{Q}$-generic over $P[g_1]$ such that, letting $g' = g_1 \cup g'_2$, $\sigma[y_{g'}][f_{g'}]\notin \tau[g']$. $\sigma[y_{g'}][f_{g'}] = x$, since $y_g$ and $f_g$ depend only on $g\upharpoonright S$. But then $x\in (B \cap P[g'])\backslash \tau[g']$, contradicting that $\tau$ weakly captures $B$.
        \end{proof}

        Pick $p\in y$ and $q\in f$ such that $\Phi((p,q))$ extends $r_1$. Then $(\sigma,(p,q))\in \tau^*$. So $x\in \tau^*[y][f]$.
    \end{proof}

    Let
    \begin{align*}
        \tau' = \{(\sigma[y],q): \,\exists p \in y \text{ such that } (\sigma,(p,q))\in\tau^*\}.
    \end{align*} 
    $\tau'$ is definable in $P[y]$ from $\tau$ and $y$. It is clear from Claim \ref{claim for lemma on term definability} that $\tau'$ weakly captures $B$ over $P[y]$.
\end{proof}

\begin{lemma}
\label{generic extension lp-closed}
    Let $P$ be $\Gamma$-ss and $y$ be $Ea_P$-generic over $P$. Then for any $a\in P[y]$, $Lp^\Gamma(a)\subset P[y]$.
\end{lemma}
\begin{proof}    
    Let $N$ be a $\Gamma$-suitable mouse built over $P$. $N$ has a Woodin cardinal $\delta_N$ above $\delta_P$. The iteration strategy for any proper initial segment of $N|\delta_N$ restricted to trees above $\delta_P$ is in $\bm{\Delta}$. And no initial segment of $N$ above $\delta_N$ projects strictly below $\delta_N$. It follows that any cardinal of $P$ remains a cardinal in $N$. In particular, $\delta_P$ remains Woodin in $N$ and $y$ is also $Ea_P$-generic over $N$.

    Suppose $R$ is an $\omega$-sound $a$-premouse with an $\omega_1$-iteration strategy in $\bm{\Delta}$ such that $R$ projects to $a$. It suffices to show $R\in P[y]$.

    Let $\alpha$ be the height of $R$. Iterating $N$ above $R$ if necessary, we may assume there is a real $g$ which is $Col(\omega,\delta_N)$-generic over $N$ such that some real in $N[y][g]$ codes $R$. By Lemma \ref{term works in generic extension}, there is a $Col(\omega,\delta_N)$-term $\tau$ in $N[y]$ which weakly captures $G_0$. Then $R$ is the unique premouse in $N[y][g]$ of height $\alpha$ such that if $x_a$ codes $a$ and $x_R$ codes $R$, then $(x_a,x_R)\in \tau[g]$. By homogeneity of the forcing, for any $g'$ which is $Col(\omega,\delta_N)$-generic over $N$, there is a premouse $R'\in N[y][g']$ of height $\alpha$ and reals $x_a$ and $x_{R'}$ in $N[y][g']$ coding $a$ and $R'$, respectively, such that $(x_a,x_{R'})\in\tau[g']$. The uniqueness of $R$ implies $R\in N[y]$. Since $R$ is coded by a subset of $a$, $R\in P[y]$.
\end{proof}

\begin{corollary}
\label{lp definable in generic ext}
    If $P$ is $\Gamma$-ss and $y$ is $Ea_P$-generic over $P$, then the map with domain $P[y]|\nu_P$ defined by $a \mapsto Lp^\Gamma(a)$ is definable in $P[y]$ from $\tau^P_{0,\nu_P}$ and $y$ (uniformly in $P$ and $y$).
\end{corollary}
\begin{proof}    
    $P[y]$ is $Lp^\Gamma$-closed by Lemma \ref{generic extension lp-closed}. Then by Lemma \ref{Lp is definable}, the map $a\to Lp^\Gamma(a)$ with domain $P[y]|\nu_P$ is definable from any term $\tau\in P[y]^{Col(\omega,\nu_P)}$ which weakly captures $G_0$ over $P[y]$. 
    
    Lemma \ref{term works in generic extension} shows there is a term $\tau\in P[y]^{Col(\omega,\nu_P)}$ which weakly captures  $G_0$ over $P[y]$ and is definable from $\tau^P_{0,\nu_P}$ and $y$ in $P[y]$.
\end{proof}

\begin{corollary}
    Suppose $P$ is $\Gamma$-ss, $y$ is $Ea_P$-generic over $P$, and $N\in P[y]|\nu_P$ is $\Gamma$-suitable. Then $\{\tau^N_{n,\mu}: \mu \text{ is an uncountable cardinal of } N \}$ is definable in $P[y]$ from $N$, $y$, and $\tau^P_{n,\nu_P}$ (uniformly in $P$, $y$, and $N$). 
\end{corollary}
\begin{proof}
    This is by the proof of Lemma \ref{terms definable}, using from Lemma \ref{term works in generic extension} that there is a term in $P$ which weakly captures $G_n$ over $P[y]$ and is definable from $\tau^P_{n,\nu_P}$ and $y$.
\end{proof}
\vspace{2mm}
\subsection{Internalizing the Direct Limit}
\label{internalization section}

Let $x_0\in\mathbb{R}$ be any real which is Turing above some real coding $W$ and consider some $M$ which is a countable, complete iterate of $M_{x_0}$.   For elements of $M|\nu_M$, being a $\Gamma$-suitable premouse, a $\Gamma$-short iteration tree, or a $\Gamma$-maximal iteration tree is definable over $M$ from $\tau^M_{0,\nu_M}$ (This follows easily from Corollary \ref{lower part is definable}). Let
\begin{align*}
    \mathcal{I}^M = \{P \in M|\nu_M: P \in\mathcal{I}\}. 
\end{align*}

\begin{lemma}
\label{M knows branches through short trees}
    Let $\T \in M|\nu_M$ be a $\Gamma$-short tree on some $\Gamma$-suitable $P\in M$. Then the branch $b$ picked by the iteration strategy for $P$ is in $M$ and $b$ is definable in $M$ from $\T$ and $\tau^M_{0,\nu_M}$ (uniformly). In particular, $M_b^\T$ and the iteration map $i_b^\T:P\to M_b^\T$ are definable in $M$ from $\T$ and $\tau^M_{0,\nu_M}$.
\end{lemma}
\begin{proof}
    Let $g$ be $Col(\omega,\nu_M)$-generic over $M$. Note $b$ is the unique branch through $\T$ which absorbs $\Q(\T)$. So by Shoenfield absoluteness, $b \in M[g]$ (in $M[g]$ the existence of such a branch is a $\Sigma^1_2$ statement about reals). But $b$ is independent of the generic $g$, so $b\in M$.
    
    It then follows from Corollary \ref{lower part is definable} that $b$, and therefore also $M_b^\T$ and $i^\T_b$, are definable in $M$ from $\tau^M_{0,\nu_M}$.
\end{proof}

\begin{corollary}
\label{gamma-guided definable}
    Suppose $P \in \mathcal{I}^M$ and $\Sigma$ is the iteration strategy for $P$. Suppose also $\T \in M|\nu_M$ is an iteration tree on $P$ below $\delta_P$ of limit length. Whether $\T$ is according to $\Sigma$ is definable in $M$ from parameter $\tau^M_{0,\nu_M}$ by a formula independent of $\T$ and the choice of $\Gamma$-ss mouse $M$.
\end{corollary}

\begin{lemma}
\label{coiterate in M|nu}
    Suppose $P,Q\in \mathcal{I}^M$. Then there is $R \in \mathcal{I}^M$ and normal iteration trees $\T$ and $\U$ on $P$ and $Q$, respectively, such that

    \begin{enumerate}
        \item $\T$ realizes $R$ is a complete iterate of $P$,
        \item $\U$ realizes $R$ is a complete iterate of $Q$,
        \item $\T\upharpoonright lh(\T) \in M|\nu_M$,
        \item $\U\upharpoonright lh(\U) \in M|\nu_M$, and
        \item $R$ is definable in $M$ from $P$, $Q$, and $\tau^M_{0,\nu_M}$ (uniformly).
    \end{enumerate}
\end{lemma}
\begin{proof}
    We perform a coiteration of $P$ and $Q$ inside $M$. Suppose so far from the coiteration we have obtained iteration trees $\T$ and $\U$ on $P$ and $Q$, respectively. 
    
    Suppose $\T$ and $\U$ have successor length. Let $P'$ and $Q'$ be the last models of $\T$ and $\U$, respectively. First consider the case $P' \trianglelefteq Q'$ or $P' \trianglelefteq Q'$. If either is a proper initial segment of the other, or there are any drops on the branches to $P'$ or $Q'$, we have violated the Dodd-Jensen property. So $P'=Q'$ and $P'$ is a common, complete iterate of $P$ and $Q$. 
    Otherwise, we continue the coiteration as usual by applying the extender at the least point of disagreement between the last models of $\T$ and $\U$, respectively.

    Now suppose $\T$ and $\U$ are of limit length. In this case $\M(\T) = \M(\U)$. If $\T$ is $\Gamma$-short, so is $\U$, and by Lemma \ref{M knows branches through short trees}, $M$ can identify the branches the iteration strategies for $P$ and $Q$ pick through $\T$ and $\U$, respectively. So the coiteration can be continued inside $M$. Otherwise, $\T$ and $\U$ are $\Gamma$-maximal. In this case let $R$ be the unique $\Gamma$-suitable mouse extending $\M(\T)$. $R$ is just the result of applying $Lp^\Gamma$ to $\M(\T)$ $\omega$ times, so $M$ can identify $R$ by Lemma \ref{Lp is definable}. Then $R$ is a complete iterate of $P$ and $Q$.

    The proof of the Comparison Lemma gives the coiteration terminates in fewer than $\nu_M$ steps. Then the argument above implies the trees from this coiteration, without their last branches, are in $M|\nu_M$ and definable in $M$.
\end{proof}

The lemma implies $\I^M$ is a directed system. $\I^M$ is countable and contained in $\I$, so we may define the direct limit $\mathcal{H}^M$ of $\I^M$, and $\mathcal{H}^M \in \I$. Let
\begin{align*}
    \tilde{\mathcal{I}}^M =\, \{P \in \mathcal{I}^M : &\text{ there is a normal iteration tree } \T \text{ such that } \T \text{ realizes } \\
    & P \text{ is a complete iterate of } W \text{ and } \T\upharpoonright lh(\T) \in M|\nu_M\}.
\end{align*}.
$\tilde{\mathcal{I}}^M$ is definable in $M$ by Corollary \ref{gamma-guided definable}.

\begin{lemma}
\label{cofinal in I^M}
    $\tilde{\mathcal{I}}^M$ is cofinal in $\mathcal{I}^M$. In particular, the direct limit of $\tilde{\mathcal{I}}^M$ is $\mathcal{H}^M$.
\end{lemma}
\begin{proof}
    Suppose $P\in \mathcal{I}^M$. By Lemma \ref{coiterate in M|nu}, there is $R \in \mathcal{I}^M$ which is a common, complete, normal iterate of both $P$ and $W$ by trees which are in $M$ (modulo their final branches). Then $R$ is below $P$ in $\mathcal{I}^M$ and $R\in \tilde{\mathcal{I}}^M$.
\end{proof}

\begin{lemma}
\label{M can approximate iteration maps}
    Suppose $P\in \mathcal{I}^M$. Let $\Sigma$ be the (unique) iteration strategy for $P$. Suppose $\T\in M|\nu_M$ is an iteration tree on $P$ according to $\Sigma$. Let $b = \Sigma(\T)$ and let $Q = M^\T_b$. Then $Q$ is definable in $M$ from $\T$ and $\tau^M_{0,\nu_M}$. And $\pi_{P,Q}\upharpoonright \gamma^P_n$ is definable in $M$ from $\T$ and $\langle\tau^M_{k,\nu_M}: k < n \rangle$ (uniformly).
\end{lemma}
\begin{proof}
    If $\T$ is $\Gamma$-short, then this is by Lemma \ref{M knows branches through short trees}.
    
    Suppose $\T$ is $\Gamma$-maximal. Then $Q = \bigcup_{i < \omega} Q_i$, where $Q_0 = \M(\T)$ and $Q_{i+1} = Lp^\Gamma(Q_i)$. So $Q$ is definable from $\M(\T)$ and $\tau^M_{0,\nu_M}$ by Corollary \ref{lower part is definable}. And $\pi_{P,Q}\upharpoonright \gamma^P_n = \pi_c\upharpoonright\gamma^P_n$, where $c$ is any branch through $\T$ respecting $\vec{G}_n$. The argument of Lemma \ref{M knows branches through short trees} shows there is a branch $c$ in $M$ respecting $\vec{G}_n$. Then $\pi_{P,Q}\upharpoonright \gamma^P_n = \pi_c\upharpoonright\gamma^P_n$ for any wellfounded branch $c\in M$ through $\T$ such that $\pi_c(\langle\tau^P_k: k < n \rangle) = \langle\tau^Q_k: k < n \rangle$. $\langle\tau^P_k: k < n \rangle$ and $\langle\tau^Q_k: k < n \rangle$ are definable in $M$ from $P$, $Q$, and $\langle\tau^M_{k,\nu_M}: k < n \rangle$ by Lemma \ref{terms definable}. So $\pi_{P,Q}\upharpoonright \gamma_n^P$ is definable in $M$ from $\T$ and $\langle\tau^M_{k,\nu_M}: k < n \rangle$.
\end{proof}

It follows from the previous lemmas that for any $P\in \mathcal{I}^M$, $\pi_{P,\mathcal{H}^M}\upharpoonright \gamma^P_n$ is definable in $M$ from $P$ and $\langle\tau^M_{k,\nu_M}: k < n \rangle$ (uniformly in $M$). The same lemmas hold in $M[y]$ for $y$ $Ea_M$-generic over $M$. In particular, we have:

\begin{lemma}
\label{M[y] can approximate branches through maximal trees}
    Suppose $y$ is $Ea_M$-generic over $M$ and $P \in \mathcal{I} \cap M[y]|\nu_M$. Let $\Sigma$ be the (unique) iteration strategy for $P$. Suppose $\T\in M[y]|\nu_M$ is an iteration tree on $P$ according to $\Sigma$. Let $b = \Sigma(\T)$ and let $Q = M^\T_b$. Then $Q$ is definable in $M[y]$ from $\T$ and $\tau^M_{0,\nu_M}$. And $\pi_{P,Q}\upharpoonright\gamma^P_n$ is definable in $M[y]$ from $\T$ and $\langle \tau^M_{k,\nu_M}: k < n\rangle$ (uniformly). Moreover, the definition is independent not just of the choice of $\Gamma$-ss mouse $M$, but also of the generic $y$.
\end{lemma}

\begin{lemma}
\label{coiterate all possible generics}
    Suppose $p\in Ea_M$ and $\dot{S}$ is an $Ea_M$-name in $M|\nu_M$ such that $p \Vdash_{Ea_M}$ ``$\dot{S}$ is a complete iterate of $W$.'' Then there is $R\in\tilde{I}^M$ such that $R$ is a complete iterate of $S[y]$ for every $y \in \R$ which is $Ea_M$-generic over $M$. Moreover, we can pick $R$ such that $R$ is (uniformly) definable in $M$ from parameters $\dot{S}$ and $p$.
\end{lemma}
\begin{proof}
    Let $\mathbb{P}$ be the finite support product $\Pi_{j<\omega} \mathbb{P}_j$, where each $\mathbb{P}_j$ is a copy of the part of $Ea_M$ below $p$. Let $H$ be $\mathbb{P}$-generic over $M$. We can represent $H$ as $\Pi_{j<\omega} H_j$, where $H_j$ is $\mathbb{P}_j$-generic over $M$. Let $\dot{S}_j$ be a $\mathbb{P}$-name for $\dot{S}[H_j]$. Let $S_j = \dot{S}_j[H]$ for $j\in \omega$ and $S_{-1} = W$.

    Lemma \ref{M[y] can approximate branches through maximal trees} tells us that $M[H]$ can perform the simultaneous coiteration of all of the $S_j$ for $j\in [-1,\omega)$ (except possibly finding the last branches). The proof of the Comparison Lemma gives that this coiteration terminates after fewer than $\nu_M$ steps. Let $R_j$ be the last model of the iteration tree on $S_j$ produced by the coiteration. Since each $S_j$ is a complete iterate of $W$, the Dodd-Jensen property implies there are no drops on the branches from $S_j$ to $R_j$ and $R_j = R_i$ for all $i,j\in[-1,\omega)$. Let $R = R_j$ for some (equivalently all) $j\in [-1,\omega)$. Then $R$ is a complete iterate of $M$ and $R$ is a complete iterate of $S_j$ for each $j\in\omega$. Let $\mathcal{U}$ be the iteration tree on $W$ from the coiteration.
    
    \begin{claim}
    \label{R ind of generic}
        $R$ is independent of the choice of generic $H$.
    \end{claim}
    \begin{proof}
        Code $R$ by a set of ordinals $X$ contained in $\nu_M$. Let $\dot{X}$ be a name for $X$. If $R$ is not independent of $H$, then there is $\alpha < \nu_M$ and $q_1,q_2\in\mathbb{P}$ such that $q_1 \Vdash\check{\alpha}\in \dot{X}$ and $q_2 \Vdash \check{\alpha}\notin \dot{X}$.

        Let $N > max(support(q_2))$. Let $\bar{q}_1$ be the condition $q_1$ shifted over by $N$ --- that is, $support(\bar{q}_1) = \{j\in [N,\omega) :\, j-N \in support(q_1)\}$ and for $j\in support(\bar{q}_1)$, $\bar{q_1}(j) = q_1(j-N)$. So $\bar{q}_1$ is compatible with $q_2$ and by symmetry, $\bar{q}_1 \Vdash \check{\alpha}\in\dot{X}$. But then there is $r \leq q_2,\bar{q}_1$ which forces both $\check{\alpha}\in\dot{X}$ and $\check{\alpha}\notin \dot{X}$.
    \end{proof}

    \begin{claim}
    \label{U ind of generic}
        $\U \upharpoonright lh(\U)$ is independent of the choice of generic $H$.
    \end{claim}
    \begin{proof}
        The same proof as in Claim \ref{R ind of generic} works.
    \end{proof}

    Claim \ref{R ind of generic} implies $R\in M|\nu_M$ and $R$ is a complete iterate of $S[y]$ for any $y$ which is $Ea_M$-generic over $M$. Claim \ref{U ind of generic} gives that $\U \upharpoonright lh(\U) \in M|\nu_M$ and thus $R\in\tilde{\mathcal{I}}^M$. 
\end{proof}
\vspace{2mm}
\subsection{The StrLe Construction}
\label{strle lemmas section}

Recall the mouse operator $x \to M_x$ defined in Section \ref{suitable mice section}. In the following lemmas let $z,x \in \mathbb{R}$ be such that $z\in M_x$ and let $M= M_x$.

\begin{lemma}
    Suppose $P = StrLe[M,z]$. Then $P$ is $\Gamma$-ss and $\delta_P = \delta_M$.
\end{lemma}
\begin{proof}
    Let $\delta = \delta_M$. By Lemma \ref{s-const lemma}, the cardinals of $P$ above $\delta$ are the same as the cardinals of $M$ and $\nu_M$ is inaccessible in $P$. Any inaccessible of $P$ above $\delta$ is inaccessible in $M$, since $M$ is a generic extension of $P$ by a $\delta$-c.c. forcing. In particular, $\nu_M$ is the unique inaccessible of $P$ above $\delta$. Then it suffices to show the following claim.
    
    \begin{claim}
        \begin{enumerate}[(a)]
            \item \label{le lp-closed} If $\eta < \delta$, then $Lp^\Gamma(P|\eta) \triangleleft P$.
            \item \label{delta is gamma-woodin} $\delta$ is a $\Gamma$-Woodin of $P$. That is, $\delta$ is Woodin in $Lp^\Gamma(P|\delta)$.
            \item \label{contained in lp} If $\eta \in P$ and $\eta \geq \delta$, then $P|(\eta^+)^P \trianglelefteq Lp^\Gamma(P|\eta)$.
            \item \label{delta is woodin} $P\models \delta$ is Woodin.
            \item \label{no woodins below delta} If $\eta < \delta$, $\eta$ is not Woodin in $Lp^\Gamma(P|\eta)$.
            \item \label{strle lp-closed} If $\eta\in P$ and $\delta \leq \eta$, then $Lp^\Gamma(P|\eta) \subseteq P$.
        \end{enumerate}
    \end{claim}
    \begin{proof}
        To prove \ref{le lp-closed}, it suffices to show if $\eta < \delta$, $R\triangleleft Lp^\Gamma(P|\eta)$, and $\rho_\omega(R) = \eta$, then $R\triangleleft P$. Coiterate $R$ against $Le[M,z]$. Suppose $\T$ and $\U$ are the iteration trees on $R$ and $Le[M,z]$, respectively, from the coiteration. $\T$ is above $\eta$ because $Le[M,z]|\eta = R|\eta$ and $\eta$ is a cutpoint of $R$. Let $\lambda < lh(\T)$ be a limit ordinal and $Q= \Q(\T) = \Q(\U)$. Since $R \in Lp^\Gamma(P|\eta)$ and $\T$ is above $\eta$, $Q \in Lp^\Gamma(\M(\T))$. $[0,\lambda]_T$ and $[0,\lambda]_U$ are the unique branches through $\T$ and $\U$, respectively, which absorb $Q$. By Corollary \ref{lower part is definable}, these branches can be identified in $M$. In particular, the coiteration of $R$ and $Le[M,z]$ can be performed in $M$. Theorem \ref{universality of m-s} gives that $R$ cannot outiterate $Le[M,z]$. Then since $R$ is $\omega$-sound, $R$ projects to $\eta$, and $Le[M,z]$ does not project to $\eta$, $R$ is a proper initial segment of $Le[M,z]|(\eta^+)^{Le[M,z]}$. $Le[M,z]$ agrees with $P$ up to $\delta$, so $R\triangleleft P$.
    
        \ref{delta is gamma-woodin} is by the proof of Theorem 11.3 of \cite{msbook}. %I think?%
        For \ref{contained in lp}, the iteration strategies for initial segments of $P|(\eta^+)^P$ restricted to iteration trees above $\delta$ are in $\bm{\Delta}$ by Fact \ref{definiability of s-const strat}. \ref{delta is woodin} is immediate from \ref{delta is gamma-woodin} and \ref{contained in lp}. See Sublemma 7.4 of \cite{hacm} for a proof of \ref{no woodins below delta}.
        
        Towards \ref{strle lp-closed}, let $Q = Lp^\Gamma(P|\eta)$. Let $\mathbb{P}$ be the extender algebra in $P$ at $\delta$ with $\delta$ generators. $M|\delta$ is $\mathbb{P}$-generic over $P$. Note $\delta$ is Woodin in $Q$ by \ref{delta is gamma-woodin}. In particular, $\mathbb{P}$ is also $\delta$-c.c. in $Q$, so any antichain of $\mathbb{P}$ in $Q$ is also in $P$ and $M|\delta$ is also $\mathbb{P}$-generic over $Q$.
        
        Let $B \in Lp^\Gamma(P|\eta)$. $B$ is in $M = P[M|\delta]$ since $P|\eta$ is in $M$ and $M$ is closed under $Lp^\Gamma$. So let $\dot{B}$ be a $\mathbb{P}$-name in $P$ such that $\dot{B}[M|\delta] = B$.
        
        Choose $p\in\mathbb{P}$ such that $p \Vdash^Q_{\mathbb{P}} \dot{B} = \check{B}$.
        
        Any $G$ which is $\mathbb{P}$-generic over $P$ is also $\mathbb{P}$-generic over $Q$. So for any $G$ which is $\mathbb{P}$-generic over $P$ such that $p\in G$, $\dot{B}[G] = B$. But then $B$ is in $P$, since $B = \{\xi < \delta : p \Vdash^P_{\mathbb{P}} \check{\xi} \in \dot{B}\}$.\footnote{Viewing $B$ as a subset of $\delta$.}
    \end{proof}
\end{proof}

\begin{lemma}
\label{another lemma about terms definable}
    Suppose $P = StrLe[M,z]$. Let $\mu \geq \delta_P$ be a cardinal of $P$. $\tau^P_{n,\mu}$ is definable in $M$ from $\tau^M_{n,\mu}$ and $z$.
\end{lemma}
\begin{proof}
    Let
    \begin{align*}
        \tau = \, \{(\sigma,p)\in P : \, \sigma \text{ is a } Col(\omega,\mu) \text{-standard term for a real}, p\in Col(\omega,\mu), \\ \text{ and } p \Vdash^M_{Col(\omega,\mu)} \sigma \in P[\dot{g}] \cap \tau^M_{n,\mu}\}.
    \end{align*}

    Since $\tau^M_{n,\mu}\in M$ and $P$ is definable over $M$ from $z$, $\tau \in M$ and is definable from $\tau^M_{n,\mu}$ and $z$. Then it suffices to show the following claim.

    \begin{claim}
        $\tau = \tau^P_{n,\mu}$
    \end{claim}
    \begin{proof}
        Clearly $\tau \subseteq \tau^P_{n,\mu}$.
        
        Suppose $(\sigma,p)\in \tau$. Let $\C$ be the set of $g$ which are $Col(\omega,\mu)$-generic over $M$ such that $p\in g$. For any $g\in\C$, $\sigma[g] \in \tau^M_{n,\mu}[g]$. In particular, $\sigma[g]\in G_n$. Since $\C$ is comeager in the set of $Col(\omega,\mu)$-generics over $P$ which extend $p$, $\sigma\in \tau^P_{n,\mu}$.
    \end{proof}
\end{proof}

\begin{lemma}
\label{strle strat f-p and guided}
    Suppose $P = StrLe[M,z]$. The iteration strategy for $P$ is fullness-preserving and guided by $\G$.
    %Do I also need it condenses well? I don't think so ...
\end{lemma}
\begin{proof}
    Let $\Sigma$ be the (unique) iteration strategy for $P$. The proof of Theorem \ref{iterability of m-s} gives that $\Sigma$ is determined by lifting an iteration on $P$ to one on $M$. More precisely, if $\T$ is a non-dropping\footnote{We leave to the reader the task of proving the case where $\T$ drops, as well as showing that $\Sigma$ condenses well.} iteration tree on $P_0 = P$ with $\langle P_\alpha \rangle$ the models of the iteration and $i_{\beta,\alpha}$ the associated iteration maps for $\beta <_T \alpha$, then we maintain an iteration tree $\T^*$ on $M_0 = M$ with models $\langle M_\alpha \rangle$ and associated iteration embeddings $i^*_{\beta,\alpha}$. We also maintain embeddings $\pi_\alpha: P_\alpha \to StrLe[M_\alpha,z]$ such that $\pi_\alpha \circ i_{\beta,\alpha} = i^*_{\beta,\alpha} \circ \pi_\beta$ and $\pi_0 = id$. In particular, $\pi_\alpha \circ i_{0,\alpha} = i^*_{0,\alpha}$.
    
    Suppose $\mu$ is a cardinal of $P$ and $\mu > \delta_P$. By Lemma \ref{another lemma about terms definable}, $i^*_{0,\alpha}(\tau^P_{n,\mu}) = \tau^{StrLe[M_\alpha,z]}_{n,\mu}$ for each $n<\omega$. Then $\pi_\alpha \circ i_{0,\alpha}(\tau^P_{n,\mu}) = \tau^{StrLe[M_\alpha,z]}_{n,\mu}$. Then by Lemma \ref{term relation condensation}, $P_\alpha$ is $\Gamma$-ss and $\pi_\alpha(\tau^{P_\alpha}_{n,\mu}) = \tau^{StrLe[M_\alpha,z]}_{n,\mu}$. This gives $\Sigma$ is fullness-preserving. A second application of Lemma \ref{term relation condensation} gives $i_{0,\alpha}(\tau^P_{n,\mu}) = \tau^{P_\alpha}_{n,\mu}$. So $\Sigma$ is guided by $\G$.
\end{proof}

\begin{corollary}
\label{strle strategy not in gamma}
    Suppose $P = StrLe[M,z]$. Then the $\omega_1$-iteration strategy for $P$ is not in $\bm{\Gamma}$.
\end{corollary}
\begin{proof}
    Immediate from Lemmas \ref{guided strat not in gamma} and \ref{strle strat f-p and guided}.
\end{proof}

\begin{lemma}
\label{iterate of mitchell-steel}
    Suppose $x,z\in\mathbb{R}$ and $x$ codes a mouse $N$ which is a complete iterate of $M_z$. Let $P = StrLe[M_x,z]$. Then $P$ is a complete iterate of $N$ below $\delta_N$.
\end{lemma}
\begin{proof}

    Coiterate $N$ and $P$. Let $\T$ and $\U$ be the iteration trees on $N$ and $P$, respectively, from the coiteration. Let $N^*$ and $P^*$ be the last models of $\T$ and $\U$, respectively.

    Suppose $P$ outiterates $N$. One possibility is that there is a drop on the branch of $\T$ from $P$ to $P^*$. Since the iteration strategy for $P$ is fullness-preserving by Lemma \ref{strle strat f-p and guided}, $P^*$ has an $\omega_1$-iteration strategy in $\bm{\Delta}$. But the strategy for $N$ is fullness-preserving and guided by $\G$. So $N^*$ cannot have an iteration strategy in $\bm{\Delta}$, contradicting that $N^* \trianglelefteq P^*$.

    If there is no drop between $P$ and $P^*$, then $N^* \triangleleft P$. Since neither side of the coiteration drops, $N^*$ and $P^*$ are both $\Gamma$-ss. But no $\Gamma$-ss mouse can have a proper initial segment which is $\Gamma$-ss.

    An identical argument shows $N$ cannot outiterate $P$. Thus $N^* = P^*$ and $\T$ and $\U$ realize $N^*$ and $P^*$ are complete iterates of $N$ and $P$, respectively. Since there are no total extenders on $N$ above $\delta_N$, $\T$ is below $\delta_N$. Similarly, $\U$ is below $\delta_P$. Then stationarity of the Mitchell-Steel construction\footnote{See e.g. 3.23 of \cite{sihmor}.} implies that $P^* = P$. So $\T$ realizes that $P$ is a complete iterate of $N$.
\end{proof}
\vspace{2mm}
\subsection{A Reflection Lemma}
\label{reflection section}

In this section we prove a lemma that any $\Sigma_1$ statement true in $M_x$ also holds in some $N \triangleleft M_x|\kappa_{M_x}$ with the property that $StrLe[N] \triangleleft StrLe[M]$. A thorough reader not already familiar with the fully-backgrounded Mitchell-Steel construction may wish to review Section \ref{m-s construction section} before proceeding. A lazy one may read the statement of Lemma \ref{reflecting below kappa} and skip to Section \ref{main thm section}.

First, we need to show $M_x$ can compute the iteration strategies of its own initial segments below its Woodin cardinal. More precisely, we have:

\begin{lemma}
\label{knows initial segments iterable}
    Let $x\in\mathbb{R}$, $N \triangleleft M_x|\delta_{M_x}$ and $\T\in M_x$ be an iteration tree on $N$ of limit length $< \delta_{M_x}$, according to the (unique) iteration strategy for $N$. The cofinal branch $b$ through $\T$ determined by the iteration strategy for $N$ is definable in $M_x$ (uniformly in $N$ and $\T$, from the parameter $\tau^{M_x}_{0,\nu_{M_x}}$). 
\end{lemma}
\begin{proof}
    Let $M = M_x$. By Corollary \ref{lower part is definable}, the function $a\mapsto Lp^\Gamma(a)$ with domain $M|\delta_M$ is definable in $M$ from the parameter $\tau^M_{0,\nu_M}$.
    
    Let $N$ and $\T$ be as in the statement of the lemma. Let $S = \M(\T)$. Clearly $S$ is definable from $\T$. Let $Q = \Q(\T)$. $Q$ is an initial segment of $Lp^\Gamma(S)$. The previous paragraph implies $Q$ is definable in $M$ from $S$ and $\tau^M_{0,\nu_M}$. The branch $b$ through $\T$ chosen by the iteration strategy for $N$ is the unique branch which absorbs $Q$.
    
    It remains to show $b$ is in $M$. Iterate $M$ to $M'$ well above where $\T$ is constructed to make some $g$ generic over $Ea^{M'}_{\delta_{M'}}$ so that $g$ codes $b$. $M'[g]$ satisfies that $b$ is the unique branch which absorbs $Q$. Since $b$ is in fact the unique such branch in $V$, symmetry of the forcing gives $b$ is in $M'$. But the iteration from $M$ to $M'$ does not add any subsets of $lh(\T)$, so in fact $b$ is in $M$.
\end{proof}

We need to put down a few more properties of the Mitchell-Steel construction before proving the main lemma of this section.

\begin{lemma}
\label{club of tau so m-s nice}
    Suppose $N$ is a mouse with a Woodin cardinal $\delta_N$. Let $z\in N\cap \R$. There is a club $C$ of $\tau < \delta_N$ such that $Le[N|\delta_N,z]|\tau = \M_\tau$, where $\M_\tau$ is the Mitchell-Steel construction of length $\tau$ in $N|\delta_N$. Moreover, we can take $C$ to be definable in $N$.
\end{lemma}
\begin{proof}
    Let $\langle \M_\xi : \xi <\delta_N\rangle$ be the models from the Mitchell-Steel construction of length $\delta_N$ over $z$, done inside $N|\delta_N$. Let $C'\subset \delta_N$ be the set of $\tau<\delta_N$ such that $\M_\tau$ has height $\tau$ and $\rho_\omega(\M_\xi) \geq \tau$ whenever $\xi$ is between $\tau$ and the height of $N$. It is not hard to see from the material in Section \ref{m-s construction section} that $C$ is a club and if $\tau\in C$, then $\M_\tau = Le[N|\delta_N,z]|\tau$.
\end{proof}

\begin{corollary}
\label{m-s is union of m-s of initial segments}
    Let $N$, $z$, and $C$ be as in Lemma \ref{club of tau so m-s nice}. Let $S$ be the set of inaccessibles of $N$ below $\delta_N$. Then $Le[N|\delta_N,z] = \bigcup_{\tau\in C\cap S} Le[N|\tau,z]$. 
\end{corollary}
\begin{proof}
    Since $\delta_N$ is Woodin in $N$, $N\models$ ``$S$ is stationary.'' And $C$ is definable in $N$, so $C\cap S$ is cofinal in $\delta_N$. Since $Le[N|\delta_N,z]$ has height $\delta_N$, $Le[N|\delta_N,z] = \bigcup_{\tau\in C\cap S} Le[N|\delta_N,z]|\tau$. So it suffices to show if $\tau \in C \cap S$, then $Le[N,z]|\tau = Le[N|\tau,z]$.
    
    Let $\langle \M_\xi : \xi <\delta_N\rangle$ be the models from the Mitchell-Steel construction of length $\delta_N$ over $z$, done inside $N$. $\tau \in C$ guarantees $Le[N,z]|\tau = \M_\tau$. And by Remark \ref{m-s up to inaccessible}, $\tau \in S$ gives $\M_\tau = Le[N|\tau,z]$. So $Le[N,z]|\tau = Le[N|\tau,z]$ for $\tau \in C \cap S$. 
\end{proof}

\begin{lemma}
\label{reflecting below kappa}
Suppose $M_x\models\phi[\vec{a}, \delta_{M_x}]$ for some $\Sigma_1$ formula $\phi$, $z\in \mathbb{R} \cap M_x$, and  $\vec{a}\in\mathbb{R}^{|\vec{a}|}\cap M_x$. Then there exists $N\triangleleft M_x|\kappa_{M_x}$ such that 
\begin{enumerate}[(a)]
    \item \label{has a woodin} $N$ has one Woodin cardinal,
    \item \label{delta inaccessible} $\delta_N$ is an inaccessible cardinal of $M_x$,
    \item \label{satisfies phi} $N\models\phi[\vec{a}, \delta_N]$, and 
    \item \label{strle initial segment} $StrLe[N,z] \triangleleft StrLe[M_x,z]$.
\end{enumerate}
\end{lemma}
\begin{proof}

Denote $M_x$ by $M$. For ease of notation we will assume $z=0$. Let $\mu$ be a cardinal of $M$ above $\delta_M$ such that $M|\mu\models \phi[\vec{a}]$.
    
\begin{claim}
\label{club of tau so hull nice}
    There is a stationary set of $\tau<\delta_M$ such that $\tau$ is inaccessible in $M$ and if $\tau\leq\zeta<\delta_M$, then $\zeta$ is not definable in $M|\mu$ from parameters below $\tau$.
\end{claim}
\begin{proof}
    Work in $M$. Let $S$ be the set of inaccessible cardinals below $\delta_M$. Since $\delta_M$ is Woodin, $S$ is stationary. Define $f:S\to\delta_M$ by setting $f(\zeta)$ to be the least $\eta$ such that there is $\zeta\leq\iota<\delta_M$ definable in $M|\mu$ from parameters in $\eta$. If the claim is false, then $f$ is regressive on a stationary set. Then by Fodor's Lemma, there is a stationary set $S_0$ and $\eta<\delta_M$ such that $f''S_0=\{\eta\}$. But $cof(\delta_M)>|\eta^{<\omega}|\times\aleph_0$, so we cannot have cofinally many elements of $\delta_M$ defined by some formula and parameters from $\eta$.
\end{proof}

Fix $\tau$ as in Lemma \ref{club of tau so m-s nice} and Claim \ref{club of tau so hull nice}. Let $H = Hull^{M|\mu}(\tau)$. Let $N$ be the transitive collapse of $H$ and $\pi:N\to M|\mu$ the anti-collapse map.
    
By condensation, $N\triangleleft M|\mu$.\footnote{See Theorem 5.1 of \cite{ooimt}.} 
Clearly $N\triangleleft M|\delta_M$, $N\models\phi[\vec{a},\delta_N]$, $\tau$ is the unique Woodin of $N$, $\tau$ is inaccessible in $M$, and $\rho_\omega(N) = \tau$.

\begin{claim}
    $Le[N|\tau] \triangleleft Le[M|\delta_M]$
\end{claim}
\begin{proof}
    For $\zeta<\tau$, $Le[N|\zeta] \triangleleft Le[N|\tau] \iff Le[M|\zeta] \triangleleft Le[M|\delta_M]$ by elementarity. But $Le[N|\zeta] = Le[M|\zeta]$ for $\zeta < \tau$. So if $Le[N|\zeta]$ is an initial segment of $Le[N|\tau]$, then it is also an initial segment of $Le[M|\delta_M]$. But this implies $Le[N|\tau]\triangleleft Le[M|\delta_M]$, since by Corollary \ref{m-s is union of m-s of initial segments}, $Le[N|\tau]$ is a union of mice of the form $Le[N|\zeta]$ for $\zeta<\tau$.
\end{proof}

We have found $N \triangleleft M|\delta_M$ satisfying \ref{has a woodin}, \ref{delta inaccessible}, \ref{satisfies phi}, $\rho_\omega(N) = \delta_N$, and $Le[N|\delta_N]\triangleleft Le[M|\delta_M]$ (since $\delta_N = \tau)$. Our next step is to reflect this below $\kappa_M$. Let $F$ be a total extender in $M$ such that the strength of $F$ is greater than $On \cap N$. In particular, we have $N\triangleleft Ult(M|\delta_M,F)$.

\begin{claim}
    $Le[N|\tau] \triangleleft Le[Ult(M|\delta_M,F)]$.
\end{claim}
\begin{proof}
    $\tau$ is inaccessible in $Ult(M|\delta_M,F)$. So by Remark \ref{m-s up to inaccessible}, $Le[N|\tau]$ equals the Mitchell-Steel construction of length $\tau$ in $Ult(M|\delta_M,F)$.
    
    Suppose the claim fails. Then there is a mouse $Q$ built during the Mitchell-Steel construction in $Ult(M|\delta_M,F)$ after $Le[N|\tau]$ is constructed, such that $Q$ projects to some $\beta < \tau$. Pick such a $Q$ which minimizes $\beta$. By Lemma \ref{knows initial segments iterable}, any initial segment of $M$ below $\delta_M$ is iterable in $M$. Then $M$ has iteration strategies for $Ult(P,F)$ for any $P\triangleleft M|\delta_M$. $Q$ is a mouse built during the Mitchell-Steel construction in $Ult(P,F)$ for some $P\triangleleft M|\delta_M$, so $Q$ is also iterable in $M$. Let $Q' = \C_\omega(Q)$. Then $Q'$ is an $\omega$-sound mouse over $Le[N|\tau]|\beta$ projecting to $\beta$ which is iterable in $M$. It follows from Theorem \ref{universality of m-s} that $Le[M|\delta_M]$ outiterates $Q'$. Since both extend $Le[N|\tau]|\beta$, and $Q'$ is $\omega$-sound and projects to $\beta$, $Q'\triangleleft Le[M|\delta_M]$. But then since $\tau$ is inaccessible in $M$, $Le[M|\tau]$ has height $\tau$, and $Le[M|\tau]\triangleleft Le[M|\delta_M]$, $Q'$ is in $Le[M|\tau]$. This is a contradiction, since a subset of $\beta$ which is not in $Le[M|\tau]$ is definable over $Q'$.
\end{proof}

By elementarity of the ultrapower embedding induced by $F$, there exists $N\triangleleft M|\kappa_M$ satisfying \ref{has a woodin}, \ref{delta inaccessible}, \ref{satisfies phi}, $\rho_\omega(N) = \delta_N$, and $Le[N|\delta_N]\triangleleft Le[M|\delta_M]$. It remains to prove the following claim.

\begin{claim}
    $StrLe[N] \triangleleft StrLe[M]$.
\end{claim}
\begin{proof}
    Since $N$ projects to $\delta_N$, so does $StrLe[N]$ (by Lemma \ref{s-const lemma}). And $StrLe[N]$ agrees with $StrLe[M]$ up to $\delta_N$ since $Le[N|\delta_N] \triangleleft Le[M|\delta_M]$. So it suffices to show $StrLe[M]$ outiterates $StrLe[N]$. But $StrLe[N]$ has an iteration strategy in $\bm{\Gamma}$, and $StrLe[M]$ cannot by Lemma \ref{strle strategy not in gamma}.
\end{proof}
\end{proof}

\subsection{Main Theorem}
\label{main thm section}

We are ready to prove Theorem \ref{main thm}. Suppose for contradiction $\langle A_\alpha |\alpha < \bm{\delta_\Gamma^+} \rangle$ is a sequence of distinct $\bm{\Gamma}$ sets. Let $U\subset \R \times \R$ be a universal $\bm{\Gamma}$ set.

Recall in Section \ref{length of direct limit section} we defined $\I$ as the direct limit of all countable, complete iterates of the $\Gamma$-suitable mouse $W$. Let $\mathcal{J} = \{(P,\xi) : P \in \mathcal{I} \wedge \xi < \delta_P$\}. Say $(P,\xi)\leq_*(Q,\zeta)$ if $(P,\xi),(Q,\zeta) \in \mathcal{J}$ and whenever $S$ is a complete iterate of both $P$ and $Q$, $\pi_{P,S}(\xi) \leq \pi_{Q,S}(\zeta)$. By Lemma \ref{pwo is short}, the relation $\leq_*$ has length $> \bm{\delta_\Gamma^+}$. Fix $n$ such that for some (equivalently any) $P\in \mathcal{I}$, $\pi_{P,\infty}(\gamma^P_n) > \bm{\delta_\Gamma^+}$. Let $\leq'_*$ be $\leq_*$ restricted to pairs $(P,\xi)$ such that $\xi < \gamma^P_n$. Then $\leq'_*$ has length $\geq \bm{\delta_\Gamma^+}$ and $\leq'_*$ is in $J_{\beta'}(\mathbb{R})$.\footnote{This is done by similar arguments to those in Section \ref{internalization section}.} Let $B_\alpha = \{y : U_y = A_\alpha\}$. By the Coding Lemma there is a set $D$ in $J_{\beta'}(\mathbb{R})$ such that $(x,y)\in D$ implies $x$ codes a pair in the domain of $\leq'_*$ and $y\in B_{|x|_{\leq'_*}}$, and $D_x$ is nonempty for all $x$ in the domain of $\leq'_*$.

Let $z_0\in\mathbb{R}$ be such that $z_0$ codes $W$ and $D\in OD^{<\beta'}(z_0)$. Let $\mathcal{I}'$ be the directed system of all countable, complete iterates of $M_{z_0}$. Let $M'_\infty$ be the direct limit of $\mathcal{I}'$. For $M,N\in \mathcal{I}'$ and $N$ an iterate of $M$, let $\pi_{M,N}: M \to N$ be the iteration map and $\pi_{M,\infty}: M \to M'_\infty$ the direct limit map (We also used $\pi_{M,N}$ and $\pi_{M,\infty}$ for $M,N\in \I$, but this should not cause any confusion).

For $M\in \mathcal{I}'$, let $\tau^M = \tau^M_{D,\delta_M}$. There is a slight issue in that our current definitions do not obviously guarantee that $\tau^M$ is moved correctly. That is, we might have a complete iterate $N$ of $M$ such that $\pi_{M,N}(\tau^M) \neq \tau^N$. This can happen because we defined the operator $x \mapsto M_x$ so that $M_x$ is guided by $\G$, but it is possible $D\notin \G$. There is no real issue here, since we can expand $\G$ to a larger self-justifying system $\G'$ such that $D\in\G'$ and require $M_x$ be guided by $\G'$. However, we should leave the operator $x \to W_x$ as is, otherwise we risk altering our construction of $D$. This raises another minor complication, because in Sections \ref{definability section} and \ref{internalization section} we assumed our $\Gamma$-ss mouse $M$ was guided by the same self-justifying system as our $\Gamma$-suitable mouse $W$. Fortunately, the results of those sections remain true so long as $\G \subseteq \G'$, modulo increasing the number of terms required as parameters in some of the lemmas. For simplicity, in what follows we will just assume $\tau^M$ is moved correctly.

\begin{definition}
    Say $M \in \mathcal{I}'$ is locally $\alpha$-stable if there is $\xi\in M$ such that $\pi_{\mathcal{H}^M,\infty}(\xi)=\alpha$. Write $\alpha_M$ for this ordinal $\xi$.
\end{definition}

\begin{definition}
    Say $M\in \mathcal{I}'$ is $\alpha$-stable if $M$ is locally $\alpha$-stable and whenever $N\in\I'$ is a complete iterate of $M$, $\pi_{M,N}(\alpha_M) = \alpha_N$.
\end{definition}

\begin{lemma}
\label{stable mouse exists}
    For any $\alpha < \bm{\delta_\Gamma^+}$, there is an $\alpha$-stable $M\in \mathcal{I}'$.
\end{lemma}
\begin{proof}
This is essentially the same as the proof of the analogous lemma in \cite{hra}. We will show for any $P\in \mathcal{I}'$, there is an iterate of $P$ which is $\alpha$-stable.

    \begin{claim}
    \label{locally stable mouse exists}
        For any $P\in\mathcal{I}'$, there is a countable, complete iterate $R$ of $P$ which is locally $\alpha$-stable
    \end{claim}
    \begin{proof}
        Fix $S\in \mathcal{I}$ and $\zeta\in S$ such that $\pi_{S,\infty}(\zeta) = \alpha$. Let $R$ be a countable, complete iterate of $P$ such that $S$ is $Ea_R$-generic over $R$.
        
        Let $\dot{S}$ be an $Ea_R$-name for $S$ such that $\emptyset \Vdash^R_{Ea_R}$ ``$\dot{S}$ is a complete iterate of $W$.'' Applying Lemma \ref{coiterate all possible generics} yields $S'\in \mathcal{I}^R$ which is a complete iterate of $S$. Then
        \begin{align*}
            \pi_{\mathcal{H}^R,\infty} \circ \pi_{S',\mathcal{H}^R} \circ \pi_{S,S'}(\zeta) &= \pi_{S,\infty}(\zeta)\\ &= \alpha.
        \end{align*}
        In particular, $\alpha\in range(\pi_{\mathcal{H}^R,\infty})$.
    \end{proof}

    Now suppose no $M \in \mathcal{I}'$ is $\alpha$-stable. Let $\langle R_j : j < \omega\rangle$ be a sequence in $\mathcal{I}'$ such that for all $j$, $R_j$ is locally $\alpha$-stable and $R_{j+1}$ is an iterate of $R_j$, but $\pi_{R_j,R_{j+1}}(\alpha_{R_j}) \neq \alpha_{R_{j+1}}$.

    \begin{claim}
    \label{dj application}
        $\pi_{R_j,R_{j+1}}(\alpha_{R_j}) \geq \alpha_{R_{j+1}}$
    \end{claim}
    \begin{proof}
        By elementarity, $\pi_{R_j,R_{j+1}}\upharpoonright \mathcal{H}^{R_j}$ is an embedding of $\mathcal{H}^{R_j}$ into $\mathcal{H}^{R_{j+1}}$. Then the Dodd-Jensen property implies for any common, complete iterate $Q$ of $\mathcal{H}^{R_j}$ and $\mathcal{H}^{R_{j+1}}$, 
        \begin{align*}
            \pi_{\mathcal{H}^{R_{j+1}},Q} \circ \pi_{R_j,R_{j+1}}(\alpha_{R_j}) \geq \pi_{\mathcal{H}^{R_j},Q}(\alpha_{R_j}).
        \end{align*}
        Then
        \begin{align*}
            \pi_{\mathcal{H}^{R_{j+1}},\infty} \circ \pi_{R_j,R_{j+1}}(\alpha_{R_j}) &\geq \pi_{\mathcal{H}^{R_j},\infty}
            (\alpha_{R_j}) \\
            &= \alpha \\
            &= \pi_{\mathcal{H}^{R_{j+1}},\infty}(\alpha_{R_{j+1}}).
        \end{align*}
        So $\pi_{R_j,R_{j+1}}(\alpha_{R_j}) \geq \alpha_{R_{j+1}}$.
    \end{proof}

    Let $R_\omega$ be the direct limit of the sequence $\langle R_j: j<\omega\rangle$. Let $\alpha_j = \pi_{R_j,R_\omega}(\alpha_{R_j})$. Claim \ref{dj application} implies $\alpha_{j+1} < \alpha_j$ for all $j$, contradicting the wellfoundedness of $R_\omega$.
\end{proof}

Let $A$ be a maximal antichain in  $Ea_M$ such that $p\in A$ implies $p$ forces the generic $ea$ is a pair $(ea^1,ea^2)$, where $ea^1$ codes a pair $(R_{ea^1},\xi_{ea^1})$ such that there exists an iteration tree on $W$ (according to the strategy for $W$) with last model $R_{ea^1}$ and $\xi_{ea^1} < \delta_{R_{ea^1}}$.\footnote{This is first order by Corollary \ref{lp definable in generic ext} and Lemma \ref{M[y] can approximate branches through maximal trees}.} Since $Ea_M$ is $\delta_M$-c.c., $|A|^M < \delta_M$. Then the disjunction of conditions in $A$ is also a condition in $Ea_M$. Pick $p^M\in Ea_M$ to be the least condition in the constructability order of $M$ which is the disjunction of conditions in some $A$ as above, to ensure $p^M$ is definable in $M$. $p^M$ is a maximal condition forcing the property above, in that any condition $p'\in Ea_M$ which forces the same property is compatible with $p^M$.

\begin{lemma}
    There is $Q^M\in \mathcal{I}^M$ such that $p^M$ forces $Q^M$ is a complete iterate of $R_{ea^1}$. Moreover, $Q^M$ is definable in $M$ from parameter $p^M$ (uniformly in $M$).
\end{lemma}
\begin{proof}
    Apply Lemma \ref{coiterate all possible generics} to the condition $p^M$ and a name for $R_{ea^1}$.
\end{proof}

\begin{definition}
    For $\alpha$-stable $M\in \mathcal{I}'$, say $p\in Ea_M$ is $\alpha$-good if $p$ extends $p^M$ and $p$ forces 
    \newline 1. $\pi_{\check{Q}^M,\mathcal{H}^M} \circ \pi_{R_{ea^1},\check{Q}^M}(\xi_{ea^1}) = \alpha_M$ and
    \newline 2. $(ea^1,ea^2)\in \tau^M$.
\end{definition}

\begin{remark}
    If $\alpha < \bm{\delta_\Gamma^+}$, being $\alpha$-good is definable over $\alpha$-stable $M\in \mathcal{I}'$ from $\alpha_M$, $\tau^M$, and $\langle\tau^M_{k,\nu_M}: k < n \rangle$ (uniformly in $M$). This follows from Lemmas \ref{M can approximate iteration maps} and \ref{M[y] can approximate branches through maximal trees}.
\end{remark}

Let $p^M_\alpha$ be the maximal $\alpha$-good condition in $M$ which is least in the construction of $M$. Note if $M$ is $\alpha$-stable and $N$ is a complete iterate of $M$, then $\pi_{M,N}(p^M_\alpha) = p^N_\alpha$.

For $w\in\mathbb{R} \cap M$ and a $\Sigma_1$ formula $\psi(w)$, write $M \models [\psi(w)]$ to mean whenever $g$ is $Col(\omega,\delta_M)$-generic over $M$, there is a proper initial segment of $M[g]$ which is a $\langle\psi',g\rangle$-witness, where $\psi'(x)$ is a formula expressing ``$\psi(f(x))$'' for some computable function $f$ such that $f(g) = w$. Note ``$M\models [\psi(w)]$'' is $\Sigma_1$ over $M$ if $M$ is iterable.

For $\alpha$-stable $M\in \mathcal{I}'$, let $S^M_\alpha$ be the set of conditions $q$ such that there exist $N,r\in M$ satisfying
\begin{enumerate}[(a)]
    \item $N\triangleleft M|\kappa_M$,
    \item $N$ has one Woodin,
    \item $\delta_N$ is a cardinal of $M$,
    \item $q,r\in Ea_N$ and $(q,r) \Vdash^N_{Ea_N \times Ea_N} [U(ea_l,ea_r^2)],$\footnote{Here by $U$ we really mean some fixed $\Sigma_1$-formula defining $U$ in $J_{\alpha_0}(\mathbb{R})$.} and
    \item $r$ is compatible with $p^M_\alpha$.
\end{enumerate}

Let $S_\alpha = \pi_{M,\infty}(S^M_\alpha)$ for some (equivalently any) $\alpha$-stable $M\in \mathcal{I}'$. $S_\alpha$ can be viewed as an element of $P(\kappa_{M'_\infty})^{M'_\infty}$.

Let $A'_\alpha$ be the set of reals $x$ such that for any $\alpha$-stable $\bar{M}\in \mathcal{I}'$ there is a countable, complete iterate $M$ of $\bar{M}$ and $q\in M$ satisfying
\begin{enumerate}
    \item $q\in S_\alpha^M$,
    \item $x\models q$, and
    \item $x$ is $Ea_M$-generic over $M$.
\end{enumerate}

\begin{lemma}
\label{A'=A}
    $A'_\alpha = A_\alpha$
\end{lemma}

\begin{corollary}
\label{S_alpha != S_beta}
    $\alpha \neq \beta \implies S_\alpha \neq S_\beta$
\end{corollary}
\begin{proof}
    Suppose $S_\alpha = S_\beta$. Let $x\in A_\alpha$. Let $\bar{M}\in \I'$ be $\beta$-stable. There is a countable, complete iterate $\hat{M}$ of $\bar{M}$ which is also $\alpha$-stable. By Lemma \ref{A'=A}, $x\in A'_\alpha$, so there is a countable, complete iterate $M$ of $\hat{M}$ such that $q\in S_\alpha^M$, $x\models q$, and $x$ is $Ea_M$-generic over $M$. $M$ is also a complete iterate of $\bar{M}$, so we have shown $x\in A'_\beta$. Applying Lemma \ref{A'=A} again, $x\in A_\beta$. Similarly, $x\in A_\beta\implies x\in A_\alpha$, so $A_\alpha = A_\beta$ and thus $\alpha = \beta$.
\end{proof}

It suffices to show Lemma \ref{A'=A}. By the same proof as for $M_\infty$ given in Lemma \ref{pwo is long}, $\kappa_{M'_\infty} \leq \bm{\delta_\Gamma}$. Then by Corollary \ref{S_alpha != S_beta}, we have $\bm{\delta_\Gamma^+}$ distinct subsets of $\bm{\delta_\Gamma}$ in $M'_\infty$. Then the successor of $\bm{\delta_\Gamma}$ in $M'_\infty$ is the successor of $\bm{\delta_\Gamma}$ in $L(\mathbb{R})$, contradicting the following claim.

\begin{claim}
    Let $\eta = \bm{\delta_\Gamma}$. Then $(\eta^+)^{M'_\infty} < (\eta^+)^{L(\R)}$
\end{claim}
\begin{proof}
    Let $\lambda = (\eta^+)^{M'_\infty}$. Since $\lambda$ is regular in $M'_\infty$ but not measurable, Lemma \ref{measurable or cof omega} implies $\lambda$ has cofinality $\omega$ in $L(\R)$.
    
    Let $f\in L(\R)$ be a cofinal function from $\omega$ to $\lambda$. Let $\langle g_\xi : \xi < \lambda \rangle$ be a sequence of functions in $M'_\infty$ such that $g_\xi:\eta \to \xi$ is a surjection. Such a sequence exists because $M'_\infty$ satisfies $AC$. Then in $L(\R)$ we can construct from $f$ and $\langle g_\xi \rangle$ a surjection from $\eta$ onto $\lambda$.
\end{proof}

\begin{proof}[Proof of Lemma \ref{A'=A}]
    First suppose $x\in A_\alpha$. Let $\bar{M}\in \mathcal{I}'$ be $\alpha$-stable. Pick $y\in \mathbb{R}$ such that $y = (y^1,y^2)$, $D(y^1,y^2)$ holds, and $|y^1|_{\leq_*} = \alpha$. Let $z$ be a real coding $\bar{M}$ and let $P=M_{\langle x,y,z \rangle}$. Let $S = StrLe[P,z_0]$.

    \begin{claim}
    \label{x,y generic over M-S}
        $x$ and $y$ are $Ea_S$-generic over $S$.\footnote{This is a standard property of the fully-backgrounded construction - see Section 1.7 of \cite{hra}.}
    \end{claim}

    \begin{claim}
        $S$ is a complete iterate of $\bar{M}$ by an iteration below $\delta_{\bar{M}}$.
    \end{claim}
    \begin{proof}
        See Lemma \ref{iterate of mitchell-steel}.
    \end{proof}

    \begin{claim}
    \label{y satisfies alpha good condtion}
        There is $r\in Ea_S$ such that $r$ is $\alpha$-good and  $y\models r$. 
    \end{claim}
    \begin{proof}
        Note by choice of $y$, $y^1$ codes a pair $(R,\xi)$ such that $R$ is a complete iterate of $W$, $\pi_{R,\mathcal{H}^S}(\xi) = \alpha_S$, and $D(y^1,y^2)$ holds. Then there is $r \in Ea_S$ such that $y\models r$, $r$ forces $\pi_{\check{Q}^S,\mathcal{H}^S} \circ \pi_{R_{ea^1},\check{Q}^S}(\xi_{ea^1}) = \alpha_S$, and $(ea^1,ea^2)\in \tau^S$.
    \end{proof}

    \begin{claim}
    \label{x in U_y realized in S}
        There exist conditions $q,r\in Ea_S$ such that $x\models q$, $y\models r$, and $(q,r) \Vdash^S_{Ea_S \times Ea_S} [U(ea_l,ea_r^2)]$.
    \end{claim}
    \begin{proof}
        By Claim \ref{y satisfies alpha good condtion}, $y$ satisfies some $\alpha$-good condition $r$. Let $y_0$ be $S[x]$-generic such that $y_0 \models r$. Then by the definition of $\alpha$-good, $y_0 = (y_0^1,y_0^2)$ where $(y_0^1,y_0^2) \in D$ and $|y_0^1|_{\leq_*} = \alpha$. It follows that $U_{y_0^2} = A_\alpha$. So $x\in U_{y_0^2}$.
    
        \begin{subclaim}
            $S[x][y_0] \models [U(x,y_0^2)]$.
        \end{subclaim}
        \begin{proof}
            Let $g$ be $Col(\omega, \delta_S)$-generic over $S[x][y_0]$. Note $S[x][y_0][g] = S[g]$ is a $g$-mouse. By the proof of Lemma \ref{generic extension lp-closed}, $Lp^\Gamma(g)$ is contained in $S[g]$. %Or see Lemma 4.3 of "Derived Models associated to Mice".
            Let $f$ be a computable function such that $f(g) = (x,y_0^2)$ and let $U'(v)$ be a formula expressing $U(f(v))$ holds. By Lemma \ref{truth gives witness}, there is a $\langle U', g\rangle$-witness which is sound, projects to $\omega$, and has an iteration strategy in $\bm{\Delta}$. Since $Lp^\Gamma(g)\subseteq S[g]$, this witness is an initial segment of $S[g]$.
        \end{proof}
    
        We have shown $S[x][y_0] \models [U(x,y_0^2)]$ for any $y_0$ which satisfies $r$ and is $S[x]$-generic. Thus there is $q\in Ea_S$ such that $x$ satisfies $q$ and $(q,r) \Vdash [U(ea_l,ea_r^2)]$.
    \end{proof}

    We next would like to find some $N\triangleleft S|\kappa_S$ with the properties of $S$ we obtained above. Note Claims \ref{x,y generic over M-S} and \ref{x in U_y realized in S} are not first order over $S$, since $x$ and $y$ are not in $S$. So a straightforward reflection argument inside $S$ will not suffice. The point of introducing $P$ and obtaining $S$ as a construction inside $P$ is that these claims are first order in $P$. The next claim demonstrates we can perform a reflection in $P$ to obtain the desired initial segment of $S$.
    
    \begin{claim}
        There is $N\triangleleft S|\kappa_S$ such that $N$ has one Woodin, $\delta_N$ is an inaccessible cardinal of $S$, $x$ and $y$ are generic for $Ea_N$, and there exist $q,r\in Ea_N \times Ea_N$ such that $x\models q$, $y\models r$, and $(q,r) \Vdash [U(ea_l,ea_r^2)]$.
    \end{claim}
    \begin{proof}
        By Claims \ref{x,y generic over M-S} and \ref{x in U_y realized in S}, $P$ satisfies 
        \begin{enumerate}
            \item \label{x,y generic as 1st order property} $x$ and $y$ are $Ea_{StrLe[P,z_0]}$-generic over $StrLe[P,z_0]$ and
            \item \label{have good conditions as first order property} there exist conditions $q,r\in StrLe[P,z_0]$ such that $x\models q$, $y\models r$, and \\$(q,r) \Vdash^{StrLe[P,z_0]}_{Ea_{StrLe[P,z_0]} \times Ea_{StrLe[P,z_0]}} [U(ea_l,ea^2_r)]$.
        \end{enumerate}

        Both properties are $\Sigma_1$ over $P$ in parameters $x$, $y$, $z_0$, and $\delta_P$. Then we may apply Lemma \ref{reflecting below kappa} to obtain $P' \triangleleft P|\kappa_P$ such that $P'$ has one Woodin cardinal, $\delta_{P'}$ is an inaccessible cardinal of $P$, $StrLe[P',z_0]\triangleleft S$, and $P'$ satisfies properties \ref{x,y generic as 1st order property} and \ref{have good conditions as first order property}.
        
        Let $N = StrLe[P',z_0]$. Note $\delta_N = \delta_{P'}$ is an inaccessible cardinal of $S$. Then all the properties we required of $N$ are apparent except that $N\triangleleft S|\kappa_S$. Standard properties of the Mitchell-Steel construction imply that $\kappa_S \geq \kappa_P$.\footnote{Suppose $\lambda < \delta_S = \delta_P$ and $E$ is an extender on the fine extender sequence of $S$ witnessing $\kappa_S$ is $\lambda$-strong in $S$. Let $E^*$ be the background extender for $E$ on the fine extender sequence of $P$. Then $E^*$ witnesses $\kappa_S$ is $\lambda$-strong in $P$.} Then $N$ has cardinality less than $\kappa_S$ in $P$, since $N$ is contained in $P'$. Since also $N\triangleleft S$, we have $N\triangleleft S|\kappa_S$.
        %For the footnote, the resurrection map is identity at kappa_S b/c kappa_S is a cardinal of S. This is what gives E^* has critical point kappa_S.
    \end{proof}

    To get $x\in A'_\alpha$, it remains to show the following claim.

    \begin{claim}
        $r$ is compatible with $p_\alpha^S$.
    \end{claim}
    \begin{proof}
        By Claim \ref{y satisfies alpha good condtion}, $y$ satisfies some $\alpha$-good condition $p$. We may assume $p$ extends $r$. $p$ is $\alpha$-good, so by maximality $p$ is compatible with $p^S_\alpha$. Then $r$ is compatible with $p^S_\alpha$ as well.
    \end{proof}

    Now suppose $x\in A'_\alpha$. Let $M,q$ realize this (for whichever $\bar{M}\in \I'$ you please) and let $N,r$ realize $q\in S^M_\alpha$. Let $y$ be $M[x]$-generic for $Ea_M$ such that $y \models r \wedge p^M_\alpha$. Since $y \models p^M_\alpha$, $y = (y^1,y^2)$ where $U_{y^2} = A_\alpha$. Since $(x,y)\models (q,r)$, $M[x][y] \models [U(x,y^2)]$. Let $g\subset Col(\omega,\delta_M)$ be $M[x][y]$-generic. Then $M[x][y][g] = M[g]$ has an initial segment $R$ witnessing $U(x,y^2)$. By taking the least such $R$, we may assume $R$ projects to $\omega$ and hence $R \in Lp^\Gamma(g)$. It follows that $x\in U_{y^2} = A_\alpha$.
\end{proof}
\vspace{2mm}
\section{Remarks on Some Projective-Like Cases}
\label{projective-like chapter}

Here we provide a few brief comments on the problem of unreachability for projective-like cases. Section \ref{projective cases section} covers the projective pointclasses. In Section \ref{mouse sets section}, we discuss what appears to be the main obstacle to proving the rest of the following conjecture.

\begin{conjecture}
\label{full conjecture in L(R)}
    Assume $ZF + AD + DC + V=L(\R)$. Suppose $\kappa\leq \bm{\delta^2_1}$ is a Suslin cardinal and $\kappa$ is either a successor cardinal or a regular limit cardinal. Then $\kappa^+$ is $S(\kappa)$-unreachable.
\end{conjecture}
\vspace{2mm}
\subsection{The Projective Cases}
\label{projective cases section}

In the introduction, we discussed a theorem of Sargsysan solving the problem of unreachability for the projective pointclasses:

\begin{theorem}[Sargsyan]
\label{sargsyan thm 2}
    Assume $ZF + AD + DC$. Then $\bm{\delta^1_{2n+2}}$ is $\bm{\Sigma^1_{2n+2}}$-unreachable.
\end{theorem}

Our technique for proving Theorem \ref{main thm} gives another proof of Sargsyan's theorem, which we outline below. We will assume $ZF + AD + DC$ for the rest of this section.

Let $W = M_{2n+1}^\#$. Let $\mathcal{I}$ be the directed system of countable, complete iterates of $W$ and let $M_\infty$ be the direct limit of $\mathcal{I}$.

\begin{fact}
    $\kappa_{M_\infty} < \bm{\delta^1_{2n+2}}$ and $\delta_{M_\infty} > (\bm{\delta^1_{2n+2}})^+$.
\end{fact}

The iteration strategy $\Sigma$ for $W$ is guided by indiscernibles, analogously to how the iteration strategies for $\Gamma$-suitable mice are guided by terms for sets in a sjs.\footnote{More explicitly, for an appropriate sequence of indiscernibles $\langle v_i:i<\omega\rangle$, $\Sigma$ is the unique iteration strategy witnessing that $W$ is strongly $\langle v_0,..,v_i\rangle$-iterable (in the sense of \cite{otpawtdsom}) for every $i<\omega$.} \cite{otpawtdsom} covers this analysis of the iteration strategy for $W$ in detail. Also analogously to Sections \ref{definability section} and \ref{internalization section}, inside an iterate $M$ of $M^\#_{2n+1}(x_0)$ for some $x_0\in \R$ coding $W$, we can form the direct limit $\mathcal{H}^M$ of countable iterates of $W$ in $M$ and approximate the iteration maps from $W$ to $\mathcal{H}^M$. This internalization is covered in \cite{hra}.

The following fact gives us an analogue of the notion of a $\langle \phi,z\rangle$-witness.

\begin{fact}
   There is a computable function which sends a $\Sigma^1_{2n+2}$-formula $\phi$ to a formula $\phi^* = \phi^*(u_0,...,u_{2n-1},v)$ in the language of mice such that the following hold:
    \begin{enumerate}
        \item If $x\in \R$, $M$ is a countable, $\omega_1+1$-iterable $x$-premouse, $M\models ZFC$, $M$ has $2n$ Woodin cardinals $\delta_0,...,\delta_{2n-1}$, $\phi$ is a $\Sigma^1_{2n+2}$ formula, and $M\models\phi^*[\delta_0,...,\delta_{2n-1},x]$, then $\phi(x)$ holds.
        \item If $x\in \R$, $\delta_0,...,\delta_{2n-1}$ are the Woodin cardinals of $M^\#_{2n}(x)$, $\phi$ is a $\Sigma^1_{2n+2}$ formula, and $\phi(x)$ holds, then a proper initial segment of $M$ above $\delta_{2n-1}$ satisfies $ZFC$ and $\phi^*[\delta_0,...,\delta_{2n-1},x]$.
    \end{enumerate}
\end{fact}

With these tools it is not difficult to adapt our proof of Theorem \ref{main thm} into a proof of Theorem \ref{sargsyan thm 2}.

Here is a brief overview of the proof of Theorem \ref{sargsyan thm 2} in \cite{hra}. The basis of this proof is also studying the directed system $\I'$ of countable iterates of $M^\#_{2n+1}(z_0)$ for some $z_0\in\R$. Suppose $\langle A_\alpha : \alpha < \bm{\delta^1_{2n+2}}\rangle$ is a sequence of distinct $\bm{\Sigma^1_{2n+2}}$ sets. Fix a $\Pi^1_{2n+3}\backslash \Sigma^1_{2n+3}$ set $A \subset \omega$. If $n\in A$, this is witnessed in a proper initial segment of any $M_{2n+1}$-like $\Pi^1_{2n+2}$-iterable premouse $M$. Then there is a $\Sigma^1_{2n+3}$ set $A' \subset A$ consisting of, roughly speaking, all $n\in\omega$ which are witnessed in such an $M$ before some $x\in A_\alpha$ is witnessed. There is $n_0\in A'\backslash A$. This is witnessed in some proper initial segment $\bar{N}_M$ of $M|\kappa_M$ for any $M\in\I'$. A coding set $S^M$ is defined analogously to our coding sets in the proof of Theorem \ref{main thm}, but with the additional requirement that the conditions appear below $\bar{N}_M$. The coding sets are used to show a $\bm{\Sigma^1_{2n+2}}$ code for $A_\alpha$ is small generic over $M$. The contradiction is obtained from this.

The technique described in the previous paragraph is a stronger argument than the one we used for Theorem \ref{main thm}, since it gives coding sets which are uniformly bounded below the least strong cardinal. It is not clear whether a similar argument could work for inductive-like pointclasses. There is no obvious analogue of the $\Pi^1_{2n+3}\backslash \Sigma^1_{2n+3}$ set $A$ for an inductive-like pointclass $\Gamma$, since there is no universal $\Gamma\backslash\Gamma^c$ set of integers. So the proof from \cite{hra} is not applicable to inductive-like pointclasses. On the other hand, the techniques of Section \ref{ind-like chapter} are applicable to the projective pointclasses. And this yields a substantially simpler proof of Theorem \ref{sargsyan thm 2}, since it eliminates the need for a uniform bound on our coding sets.\\

\subsection{Mouse Sets and Open Problems}
\label{mouse sets section}
In this section we discuss the relationship between the problem of unreachability and well-known conjectures on mouse sets. We will assume $ZF + AD + DC + V=L(\R)$, although this is overkill for some of the results stated below.

\begin{definition}
    $X \subset \R$ is a mouse set if there is an $\omega_1+1$-iterable premouse $M$ such that $X = M \cap \R$.
\end{definition}

\begin{theorem}[Steel]
\label{projective mouse set thm}
    Suppose $\Gamma = \Sigma^1_{n+2}$ for some $n\in \omega$. Then $C_\Gamma$ is a mouse set. 
\end{theorem}

\begin{theorem}[Woodin]
\label{mouse set thm}
    Suppose $\lambda$ is a limit ordinal and let \newline $\Gamma = \{ A \subseteq \R :\, A \text{ is definable in } J_\beta(\R) \text{ for some } \beta<\lambda\}$. Then $C_\Gamma$ is a mouse set.
\end{theorem}

See \cite{pwim} and \cite{steel2016} for proofs of Theorems \ref{projective mouse set thm} and \ref{mouse set thm}, respectively. \cite{steel2016} also gives the following conjecture.

\begin{conjecture}[Steel]
\label{lightface mouse set conjecture}
    Suppose $\Gamma$ is a level of the (lightface) Levy hierarchy.\footnote{I.e. $\Gamma = \Sigma_n(J_\alpha(\R))$ for some $\alpha\in On$ and $n\in\omega$.} Then $C_\Gamma$ is a mouse set.
\end{conjecture}

Conjecture \ref{lightface mouse set conjecture} is a way of asking if there is a mouse corresponding exactly to the pointclass $\Gamma$. For each $\Gamma$ in the Levy hierarchy, the core model induction constructs a mouse which contains $C_\Gamma$, but in some cases the mouse constructed is too large. For example, let $J$ be the mouse operator $J(x) = \bigcup_{n<\omega} M^\#_n(x)$. If $\Gamma = \Sigma_{n+2}(J_2(\R))$, then
\begin{align*}
    M^{J^\#}_n \cap \R \subsetneq C_\Gamma \subsetneq M^{J^\#}_{n+1} \cap \R.
\end{align*}
There are many similar cases in which the mice constructed in \cite{cmi} skip the (hypothesized) mouse realizing Conjecture \ref{lightface mouse set conjecture}. Recent progress has been made towards Conjecture \ref{lightface mouse set conjecture} in \cite{rudominernew}, which resolves the case $\Gamma = \Sigma_2(J_2(\R))$.

The problem of unreachability is connected to a boldface version of Conjecture \ref{lightface mouse set conjecture}.

\begin{conjecture}
\label{boldface mouse set conjecture}
    Suppose $\alpha \in ON$ and $n \in \omega$. For $x\in \R$, let $\Gamma_x$ consist of all pointsets $A$ for which there is a $\Sigma_n$ formula $\phi$ with parameter $x$ such that $A = \{y : J_\alpha(\R)\models \phi[y]\}$. Then for any $y\in \R$, there is $x\in\R$ such that $y \leq_T x$ and $C_{\Gamma_x}$ is a mouse set.
\end{conjecture}

Presumably a proof of Conjecture \ref{lightface mouse set conjecture} would relativize, so a proof of Conjecture \ref{lightface mouse set conjecture} would also resolve Conjecture \ref{boldface mouse set conjecture}.

The mouse operator $x \mapsto M^\#_{2k}(x)$ realizes Conjecture \ref{boldface mouse set conjecture} holds for $\alpha=1$ and $n = 2k+2$. To prove $\bm{\delta^1_{2n+2}}$ is $\bm{\Sigma^1_{2n+2}}$-unreachable, we studied the direct limit of $M = M^\#_{2n+1}(x_0)$ for some $x_0\in\R$. Note if $g$ is $Col(\omega,\delta_M)$-generic over $M$, then $M[g] = M^\#_{2n}(g)$.

For $\alpha$ admissible, the mouse operator $x\mapsto M_x$ of Theorem \ref{nice strategy exists} realizes Conjecture \ref{boldface mouse set conjecture} holds in the case $n=1$. Note if $g$ is $Col(\omega,\delta_{M_x})$-generic over $M_x$, then $M_x[g] \cap \R = Lp^\Gamma(g) \cap \R = C_{\Gamma}(g)$. So in the inductive-like case as well we studied the direct limit of a mouse such that collapsing its least Woodin yields a mouse realizing one case of Conjecture \ref{boldface mouse set conjecture}.

Thus for each pointclass $\bm{\Sigma_n}(J_\alpha(\R))$ for which we have proven Conjecture \ref{Grigor conjecture} holds, we used a mouse operator realizing Conjecture \ref{boldface mouse set conjecture} holds for $\alpha$ and $n$. It seems likely a proof of Conjecture \ref{full conjecture in L(R)} would involve proving Conjecture \ref{boldface mouse set conjecture} for each $\alpha$ and $n$ such that $\bm{\Sigma_n}(J_\alpha(\R)) = S(\kappa)$ for some Suslin cardinal $\kappa$ which is a successor cardinal or a regular limit cardinal.

\section*{Funding}
The first two authors were supported by the National Science Foundation under grant No. DMS-1764029. The third author’s work is funded by the National Science Center, Poland under the Weave-Unisono Call, registration number UMO-2023/05/Y/ST1/00194.

\printbibliography

\end{document}